\documentclass{article}
\usepackage{graphicx} 
\usepackage{amsthm}
\usepackage{amsfonts, amsmath, amssymb, amsthm}
\usepackage{cleveref}
\usepackage[utf8x]{inputenc}
\usepackage{url}

\title{$\Pi^0_4$ conservation of Ramsey's theorem for pairs}
\date{June 2023}

\newtheorem*{statement}{Statement}
\newtheorem{theorem}{Theorem}
\numberwithin{theorem}{section}

\newtheorem{lemma}[theorem]{Lemma}
\newtheorem{question}[theorem]{Question}
\newtheorem{proposition}[theorem]{Proposition}
\newtheorem{remark}[theorem]{Remark}
\newtheorem{definition}[theorem]{Definition}
\newtheorem{corollary}[theorem]{Corollary}

\newcommand{\Param}{P}
\newcommand{\param}{p}

\makeatletter
\newtheorem*{rep@theorem}{\rep@title}
\newcommand{\newreptheorem}[2]{%
\newenvironment{rep#1}[1]{%
 \def\rep@title{#2 \ref{##1}}%
 \begin{rep@theorem}}%
 {\end{rep@theorem}}}
\makeatother

\newreptheorem{theorem}{Theorem}

\usepackage{xcolor}

\newcommand{\RCA}{\mathsf{RCA}}
\newcommand{\WKL}{\mathsf{WKL}}
\newcommand{\EM}{\mathsf{EM}}

\newcommand{\ADS}{\mathsf{ADS}}

\newcommand{\COH}{\mathsf{COH}}

\newcommand{\BSig}{\mathsf{B}\Sigma^0}
\newcommand{\ISig}{\mathsf{I}\Sigma^0}

\newcommand{\IDel}{\mathsf{I}\Delta^0}

\newcommand{\WF}{\mathsf{WF}}
\newcommand{\RT}{\mathsf{RT}}

\newcommand{\GP}{\mathsf{GP}}

\newcommand{\FGP}{\mathsf{FGP}}
\newcommand{\pFGP}[2]{(#1, #2)\mbox{-}\FGP}

\renewcommand{\L}{\mathcal{L}}
\newcommand{\C}{\mathcal{C}}
\newcommand{\R}{\mathcal{R}}

\renewcommand{\S}{\mathcal{S}}

\newcommand{\M}{\mathcal{M}}
\newcommand{\Nc}{\mathcal{N}}

\newcommand{\NN}{\mathbb{N}}

\newcommand{\uh}[0]{{\upharpoonright}}

\newcommand{\card}{\operatorname{card}}
\newcommand{\finsub}{\subseteq_{\mathtt{fin}}}

\usepackage{authblk}
\DeclareSymbolFont{bbold}{U}{bbold}{m}{n}
\DeclareMathSymbol{\bbomega}{\mathord}{bbold}{"7F}

\author{Quentin Le Hou\'erou \and Ludovic Levy Patey \and Keita Yokoyama}

\def\qt#1{``#1''}%

\begin{document}

\maketitle

\begin{abstract}
In this article\footnote{A first version of this article~\cite{houerou2026conservation} was published at the Journal of the London Mathematical Society. The authors proved the closure of largeness$(T)$ under Ramsey's theorem for pairs ($\RT^2_2$) by decomposing it into the Erd\H{o}s-Moser theorem ($\EM$) and the Ascending Descending Sequence principle $(\ADS)$. However, the proof of closure under $\ADS$ (\cite[Proposition 5.4]{houerou2026conservation}) was flawed, affecting \cite[Corollary 5.5]{houerou2026conservation}. The authors then published a corrigendum in which they proved directly \cite[Corollary 5.5]{houerou2026conservation} with different explicit bounds. The current article is a fixed version.}, we prove that Ramsey's theorem for pairs and two colors is a $\forall \Pi^0_4$ conservative extension of~$\RCA_0 + \BSig_2$, where a $\forall \Pi^0_4$ formula consists of a universal quantifier over sets followed by a $\Pi^0_4$ formula. The proof is an improvement of a result by Patey and Yokoyama~\cite{patey2018proof} and a step towards the resolution of the longstanding question of the first-order part of Ramsey's theorem for pairs.
For this, we introduce a new general technique for proving $\Pi^0_4$-conservation theorems. \bigskip

\textbf{MSC classes} Primary: 03B30, 03F30 Secondary: 05D10
\end{abstract}

\section{Introduction}

Among the theorems studied in Reverse Mathematics, Ramsey's theorem for pairs plays a significant role, as it escapes the structural phenomenon of the so-called Big Five. The study of its $\omega$-models yielded many longstanding open problems, and each of them motivated the discovery of new techniques in Computability Theory~\cite{seetapun1995strength,cholak2001strength,liu2012rt22,monin2021srt22}. See Hirschfeldt~\cite{hirschfeldt2015slicing} for a gentle introduction to the Reverse Mathematics of Ramsey's theorem. 

Given a set~$X \subseteq \NN$, we write $[X]^n$ for the set of unordered $n$-tuples over~$X$. A set~$H \subseteq \NN$ is \emph{homogeneous} for a coloring $f : [\NN]^n \to k$ if $f$ is constant over~$[H]^n$.

\begin{statement}[Ramsey's theorem for $n$-tuples and $k$ colors]
$\RT^n_k$ : For every coloring $f : [\NN]^n \to k$, there is an infinite $f$-homogeneous set.
\end{statement}

From a proof-theoretic perspective, Ramsey's theorem for pairs also raised many challenging open questions, among which the characterization of its first-order part.
The \emph{first-order part} of a theorem $T$ of second-order arithmetic is the set of its first-order consequences, that is, the sentences in the language $\mathcal{L}_{\mathsf{PA}}$ which are provable by~$T$. $\RT^2_2$ is known to imply the collection principle for $\Sigma^0_2$ formulas ($\BSig_2$) over~$\RCA_0$ (see Hirst~\cite{hirst1987combinatorics}) and to be $\Pi^1_1$ conservative over~$\RCA_0 + \ISig_2$ (see Cholak, Jockusch and Slaman~\cite{cholak2001strength}). On the other hand, Chong, Slaman and Yang~\cite{chong2017inductive} proved that $\RT^2_2$ does not imply $\ISig_2$ over~$\RCA_0$. The following question is arguably the most important open question of the reverse mathematics of Ramsey's theorem:

\begin{question}\label[question]{ques:rt22-pi11-conservation}
Is $\RT^2_2$ $\Pi^1_1$ conservative over~$\RCA_0 + \BSig_2$?
\end{question}

The reader can refer to \cite{kolo2021search} for a history of the quest for the first-order part of Ramsey's theorem for pairs.
Thanks to an isomorphism theorem for weak K\"onig's lemma ($\WKL_0$), Fiori-Carones, Ko{\l}odziejczyk, Wong and Yokoyama~\cite{fiori2021isomorphism} proved that in order to prove $\Pi^1_1$ conservation of $\RT^2_2$ over~$\RCA_0 + \BSig_2$, it is sufficient to study only a fixed level in the arithmetic hierarchy.

\begin{theorem}[Fiori-Carones et al.~\cite{fiori2021isomorphism}]
$\RT^2_2$ is $\Pi^1_1$ conservative over~$\RCA_0 + \BSig_2$
iff it is $\forall \Pi^0_5$ conservative over~$\RCA_0 + \BSig_2$.
\end{theorem}

Here, a formula is $\forall \Pi^0_n$ if it is of the form $\forall X \varphi(X)$ where~$\varphi$ is a $\Pi^0_n$ formula.
Patey and Yokoyama~\cite{patey2018proof} proved the following theorem:

\begin{theorem}[Patey and Yokoyama~\cite{patey2018proof}]
$\RT^2_2$ is $\forall \Pi^0_3$ conservative over~$\RCA_0$. 
\end{theorem}

Note that the $\forall \Pi^0_3$ consequences of~$\RCA_0$ and $\RCA_0 + \BSig_2$ coincide, by a parameterized version of the Parsons, Paris and Friedman conservation theorem (see~\cite{kaye1991models} or \cite{buss1998handbook}).
In this article, we make one further step towards the characterization of the first-order part of Ramsey's theorem for pairs, by proving the following theorem:

\begin{theorem}[Main theorem]\label[theorem]{thm:rt22-pi04-conservation}
$\WKL_0 + \RT^2_2$ is $\forall \Pi^0_4$ conservative over~$\RCA_0 + \BSig_2$.
\end{theorem}

The proof of the main theorem follows the structure of Patey and Yokoyama~\cite{patey2018proof},
using an indicator for $\forall \Pi^0_4$ conservation defined by Yokoyama~\cite[Section 4]{yokoyama2013notes}.
It is based on a variant of the notion of $\alpha$-largeness, which provides a new general technique for proving $\forall \Pi^0_4$-conservation theorems.

As in \cite{patey2018proof}, in order to avoid confusion between the theory and the meta-theory, we will use $\omega$ to denote the set of (standard) natural numbers, $\NN$ to denote the sets of natural numbers inside the system and $\bbomega$ for the ordinal $\omega$ in the system.

\subsection{Quantitative largeness and $\Pi^0_3$ sentences}

A family of finite sets of natural numbers $\L \subseteq [\NN]^{<\NN}$ is said 
to be a \textit{largeness notion} if any infinite set has a finite subset in $\L$
and $\L$ is closed under supersets.
Ketonen and Solovay~\cite{ketonen1981rapidly} defined a quantitative notion of largeness, called $\bbomega^n$-largeness (see \Cref{defi:largeness-rca0}), to measure the size of sets necessary to satisfy $\Sigma^0_1$ formulas which $\WKL_0$-provably hold over infinite sets. More precisely, the following theorem holds:

\begin{theorem}[Generalized Parsons theorem for~$\WKL_0$~\cite{patey2018proof}]\label[theorem]{thm:generalized-parsons-rca0}
Let~$\psi(x, y, F)$ be a $\Sigma_0$ formula with exactly the displayed free variables. Assume that 
$$
\WKL_0 \vdash \forall x \forall X(X \mbox{ is infinite } \rightarrow \exists F \finsub X \exists y \psi(x, y, F))
$$
Then there exists some~$n \in \omega$ such that $\ISig_1$ proves $\forall x \forall Z \finsub (x, \infty)$
$$
Z \mbox{ is } \bbomega^n\mbox{-large} \rightarrow \exists F \subseteq Z \exists y < \max Z \psi(x, y, F)
$$
\end{theorem}

The generalized Parsons theorem for $\WKL_0$ plays a key role in conservation results, as it provides quantitative finitary counterparts to infinitary theorems. See Patey and Yokoyama~\cite{patey2018proof} for a proof of \Cref{thm:generalized-parsons-rca0}.

For the purpose of $\forall \Pi^0_4$ conservation over~$\WKL_0 + \BSig_2$, we shall define a new quantitative notion of largeness, called $\bbomega^n$-largeness$(T)$ (see \Cref{defi:largeness-rca0-bsig2-t}), where $T$ is any fixed $\Pi^0_3$ sentence, and prove the following generalized Parsons theorem:

\begin{theorem}[Generalized Parsons theorem for $\WKL_0 + \BSig_2 + T$]\label[theorem]{thm:generalized-parsons-p4-bsig2-largeness}
Let~$\psi(x, y, F)$ be a $\Sigma_0$ formula with exactly the displayed free variables. Assume that 
$$
\WKL_0 + \BSig_2 + T \vdash \forall x \forall X(X \mbox{ is infinite } \rightarrow \exists F \finsub X \exists y \psi(x, y, F))
$$
Then there exists some~$n \in \omega$ such that $\ISig_1$ proves $\forall x \forall Z \finsub (x, \infty)$
$$
Z \mbox{ is } \bbomega^n\mbox{-large}(T) \mbox{ and exp-sparse } \rightarrow \exists F \subseteq Z \exists y < \max Z \psi(x, y, F)
$$
\end{theorem}

In \Cref{sect:framework-pi04}, we shall develop a framework for $\forall \Pi^0_4$ conservation over~$\BSig_2$, and will in particular relate through \Cref{thm:largeness-to-conservation} the notion of $\forall \Pi^0_4$ conservation to provability over~$\RCA_0 + \BSig_2 + T$, for a $\Pi^0_3$ sentence~$T$.

\subsection{Prerequisites}\label[section]{sect:prerequisites}

Along the paper, we shall resort to some well-known conservation theorems without explicitly stating them. 

\begin{statement}[Weak K\"onig's lemma]
$\WKL$: Every infinite binary tree admits an infinite path.
\end{statement}

To follow the standard reverse mathematical practice, we shall write $\WKL_0$ for the theory~$\RCA_0 + \WKL$.
Weak K\"onig's lemma plays a central role in Reverse Mathematics, as it captures compactness arguments.

\begin{theorem}[H\'ajek~\cite{hajek1993interpretability}]
$\WKL_0+\BSig_2$ is a $\Pi^1_1$-conservative extension of $\RCA_0 + \BSig_2$.
\end{theorem}

Ramsey's theorem for pairs admits a decomposition into two combinatorially simpler statements, namely, the Ascending Descending Sequence principle and the Erd\H{o}s-Moser theorem.

\begin{statement}[Ascending Descending Sequence]
$\ADS$: Every infinite linear order admits an infinite ascending or descending sequence.
\end{statement}

Given a coloring~$f : [\NN]^2 \to 2$, a set~$H \subseteq \NN$ is \emph{$f$-transitive} if for every~$\{x,y,z\} \in [H]^3$ with $x < y < z$ and every~$i < 2$, if $f(x,y) = f(y, z) = i$, then $f(x, z) = i$.

\begin{statement}[Erd\H{o}s-Moser theorem]
$\EM$: For every infinite coloring~$f : [\NN]^2 \to 2$, there is an infinite $f$-transitive subset.
\end{statement}

Ramsey's theorem for pairs and two colors is equivalent to $\ADS \wedge \EM$ (see Bovykin and Weiermann~\cite{bovykin2017strength}). 
Chong, Slaman and Yang~\cite[Corollary 4.4]{chong2021pi11} proved the following theorem, among other $\Pi^1_1$ conservations of combinatorial statements:

\begin{theorem}[Chong, Slaman and Yang~\cite{chong2021pi11}]\label[theorem]{thm:ads-pi11-conservation}
$\ADS$ is a $\Pi^1_1$ conservative extension of~$\RCA_0 + \BSig_2$. 
\end{theorem}

The following theorem was proven by Patey and Yokoyama~\cite{patey2018proof} and is a generalization of an amalgamation theorem for $\Pi^1_1$ conservation by Yokoyama~\cite{yokoyama2009conservativity}.

\begin{theorem}[Amalgamation~\cite{patey2018proof}]\label{thm:amalgamation}
Fix $n \ge 1$.
Let $T$ be a theory extending $\ISig_1$ which consists of sentences of the form $\forall X \exists Y \theta(X,Y)$ where $\theta$ is $\Pi^{0}_{n+2}$,
and let $\Gamma_{1}$ and $\Gamma_{2}$ be sentences of the same form as $T$.
If $T+\Gamma_{i}$ is a $\forall\Pi^{0}_{n+2}$-conservative extension of $T$ for $i=1,2$,
then, $T+\Gamma_{1}+\Gamma_{2}$ is a $\forall\Pi^{0}_{n+2}$-conservative extension of $T$.
\end{theorem}

In particular, putting \Cref{thm:ads-pi11-conservation} and \Cref{thm:amalgamation} together, it is sufficient to prove that $\EM$ is a $\forall\Pi^0_4$ conservative extension of~$\RCA_0 + \BSig_2$ to prove our main theorem. \\

Every $\Sigma^0_1$ formula $\varphi(G)$ can be expressed in normal form as $\exists k \psi(G \uh_k)$, where $\psi(\sigma)$ is a monotonous $\Sigma^0_1$ formula, that is, if $\psi(\sigma)$ and $\sigma \preceq \tau$, then $\psi(\tau)$ holds. 
All over this paper, we shall use the following standard compactness argument:

\begin{lemma}[Folklore]\label[lemma]{lem:compactness-big2-wkl}
$\WKL_0 + \BSig_2$ proves that for every $\Sigma^0_1$ formula~$\varphi(G) \equiv \exists k \psi(G \uh k)$ in normal form such that $\forall A \varphi(A)$ holds, there exists some $N_1 \in \NN$ such that for every $\sigma \in 2^{N_1}$, $\psi(\sigma)$ holds. Moreover, for every $\Delta^0_2$ predicate~$A$, $\varphi(A)$ holds.
\end{lemma}
\begin{proof}
Suppose there is no such~$N_1$. Let~$T$ be the tree of all $\sigma \in 2^{<\NN}$ such that $\neg \psi(\sigma)$. By monotonicity of~$\psi$, $T$ is a tree, and by assumption, it is infinite. By $\WKL_0$, there is some~$A \in [T]$. Then for every~$k$, $\neg \psi(A \uh k)$, so $\neg \varphi(A)$, contradicting our assumption.
Let~$N_1$ witness the statement of the lemma and let $A$ be a $\Delta^0_2$ predicate. By $\BSig_2$, $A$ is regular, so $A \uh_{N_1}$ is coded. By choice of~$N_1$, $\psi(A \uh_{N_1})$ holds, so $\exists k \psi(A \uh_k)$ holds, hence $\varphi(A)$ holds.
\end{proof}

\subsection{Structure of this paper}

In \Cref{sect:largeness-t}, we define two quantitative notions of largeness. The first one, $\bbomega^n$-largeness, was defined by Ketonen and Solovay~\cite{ketonen1981rapidly} and admits a generalized Parsons theorem for provability over~$\WKL_0$. The second one, $\bbomega^n$-largeness$(T)$, is new, and plays the same role as $\bbomega^n$-largeness, but for provability over~$\WKL_0 + \BSig_2 + T$, where~$T$ is any fixed $\Pi^0_3$ sentence.

Then, in \Cref{sect:framework-pi04}, we develop a framework for $\forall \Pi^0_4$ conservation of $\RT$-like theorems. For this, given a $\RT$-like theorem~$\Gamma$, we first define in \Cref{sect:density-ramsey-like} a notion of $n$-density$(\Gamma)$, which informally asserts the existence of sufficiently large finite sets, such that $n$ consecutive applications of~$\Gamma$ yield a large set. Following techniques initially defined by Kirby and Paris~\cite{Kirby1977InitialSO}, we prove that $\RT^2_2$ is a $\forall \Pi^0_4$ conservative extension of $\RCA_0 + \BSig_2 + $ \qt{there exists $n$-dense sets for every~$n \in \omega$} (see \Cref{thm:ph-conservative}). $\forall \Pi^0_4$ conservation over~$\RCA_0 + \BSig_2$ is then reduced to whether $\RCA_0 + \BSig_2$ proves the existence of these $n$-dense$(\Gamma)$ sets, for every~$n \in \omega$. It is however difficult to directly prove the existence of $n$-dense$(\Gamma)$ sets, and we therefore resort in~\Cref{sect:largeness-ramsey-like} to the quantitative notion of largeness, $\bbomega^n$-largeness$(T)$, to prove this existence step by step, by handling one application of~$\Gamma$ at a time.

In \Cref{sect:grouping-principle}, we introduce the grouping principle, originally defined by Patey and Yokoyama~\cite{patey2018proof}, and prove that it is a $\Pi^1_1$ conservative extension of~$\RCA_0 + \BSig_2$. This principle was shown very useful to construct $\bbomega^{n+1}$-large sets from sequences of $\bbomega^n$-large sets after an application of $\RT^2_2$ in~\cite{patey2018proof}. Using the generalized Parsons theorem for $\WKL_0 + \BSig_2 + T$, we derive a finitary version of the grouping principle in~\Cref{sect:finite-grouping-principle}. This will be the actual version which will be used in the last section.
In \Cref{sect:applications}, we apply this framework to prove that $\EM$ and $\RT^2_2$ are $\forall \Pi^0_4$ conservative extensions of~$\RCA_0 + \BSig_2$.

In \Cref{sect:weak-formulas}, we give two proofs that $\RCA_0 + \BSig_2$ is a conservative extension of~$\RCA_0$ for a restricted class of $\forall \Pi^0_4$ formulas.

In \Cref{sect:conclusion}, we prove the tightness of the upper bound of largeness for the pigeonhole principle.
Last, in \Cref{sect:open-questions}, we open the discussion to proof sizes and state some related open questions.

\section{Largeness}\label[section]{sect:largeness-t}

In this section, we define some quantitative notions of largeness both for~$\WKL_0$ and for~$\WKL_0 + \BSig_2 + T$, where $T$ is any $\Pi^0_3$ sentence. The first notion of largeness is originally due to Ketonen and Solovay~\cite{ketonen1981rapidly}, and is already well-understood. The second notion is new, so we conduct a systematic study of it, by proving some combinatorial lemmas, and a generalized Parsons theorem, which will often be used in the remainder of this article.

\subsection{Largeness for~$\WKL_0$}

Ketonen and Solovay~\cite{ketonen1981rapidly} defined a quantitative notion of largeness, called $\alpha$-largeness for every~$\alpha < \epsilon_0$, and proved a partition theorem for this notion. We give here an equivalent inductive definition in the restricted case where $\alpha$ is of the form $\bbomega^n \cdot k$ with~$n, k \in \NN$. This definition will serve as a basis for the quantitative notion of largeness for~$\WKL_0 + \BSig_2 + T$.

\begin{definition}[Largeness for~$\WKL_0$]\label[definition]{defi:largeness-rca0}
A set~$X \finsub \NN$ is 
\begin{itemize}
    \item \emph{$\bbomega^0$-large} if $X \neq \emptyset$. 
    \item \emph{$\bbomega^{(n+1)}$-large} if $X \setminus \min X$ is  $(\bbomega^n \cdot \min X)$-large
    \item \emph{$\bbomega^n \cdot k$-large} if
there are $k$ $\bbomega^n$-large subsets of~$X$
$$
X_0 < X_1 < \dots < X_{k-1}
$$
where $A < B$ means that for all $a \in A$ and $b \in B$, $a < b$.
\end{itemize}
\end{definition}

Following the convention of Ko\l odziejczyk and Yokoyama~\cite{kolo2020some}, we always assume that $\min X \geq 3$ to avoid some degenerate behavior when $\min X$ is too small.
The equivalence between \Cref{defi:largeness-rca0} and the original one is a consequence of~\cite[Theorem II.3.21]{hajek1998metamathematics}. This notion of largeness enjoys many desirable properties: $\RCA_0$ proves that $\bbomega^n$-largeness is a largeness notion for every~$n \in \omega$. Moreover, it satisfies the generalized Parsons theorem for~$\RCA_0$ (\Cref{thm:generalized-parsons-rca0}).

\begin{remark}
The definition above is also a notion of largeness for~$\WKL_0 + \BSig_2$, and \Cref{thm:generalized-parsons-rca0} also holds if one replaces provability over~$\WKL_0$ by provability over~$\WKL_0 + \BSig_2$, as $\WKL_0 + \BSig_2$ is a $\forall \Pi^0_3$ conservative extension of~$\WKL_0$ (see~\cite{kaye1991models} or \cite{buss1998handbook}).
\end{remark}

The main result of Ketonen and Solovay is the following partition theorem:

\begin{theorem}[Ketonen and Solovay~\cite{ketonen1981rapidly}]\label[theorem]{thm:ketonen-solovay}
Fix~$k \in \omega$. $\RCA_0$ proves that for every~$\bbomega^{k+4}$-large set~$X \finsub \NN$ 
and for every coloring $f : [X]^2 \to k$,
there is an $\bbomega$-large $f$-homogeneous subset~$Y \subseteq X$.
\end{theorem}

The theorem of Ketonen and Solovay was later generalized by Ko\l odziejczyk and Yokoyama~\cite[Theorem 1.6]{kolo2020some}, who proved that if~$X$ is $\bbomega^{300^{k-1}n}$-large, then for every coloring $f : [X]^2 \to k$, there is an $\bbomega^n$-large $f$-homogeneous subset.

\subsection{Largeness for $\WKL_0 + \BSig_2 + T$}

We now adapt the notion of $\bbomega^n$-largeness, to obtain a quantitative notion of largeness which will play the same role, but for provability over~$\WKL_0 + \BSig_2 + T$ where~$T$ is a $\Pi^0_3$ sentence.

Such a development might seem a bit ad-hoc and unrelated to $\forall \Pi^0_4$ conservation at first sight, but it will make full sense in~\Cref{sect:largeness-ramsey-like}, where it will be shown to be very useful for proving the existence of density notions. The reason of the use of a $\Pi^0_3$ sentence comes from the standard proof that a statement~$\Gamma$ is $\forall \Pi^0_4$ conservative over~$\RCA_0 + \BSig_2$. Assume that some sentence of the form $\forall \Param \forall \param \exists x \forall y \exists z \zeta(\Param \uh_z, \param, x, y, z)$ is not provable over~$\RCA_0 + \BSig_2$, where~$\zeta$ is a $\Sigma^0_0$-formula. Then there exists some model $\M = (M, S)$ with some set~$\Param \in S$ and some integer~$p \in M$ such that $\M \models \RCA_0 + \BSig_2 + \forall x \exists y \forall z \neg \zeta(\Param \uh_z, \param, x, y, z)$. Working in a language enriched with constants~$\Param$ and $\param$, and letting~$T \equiv \forall x \exists y \forall z \neg \zeta(\Param \uh_z, \param, x, y, z)$, we want to create a model of~$\RCA_0 + \BSig_2 + T$ which furthermore satisfies~$\Gamma$. We therefore naturally end up reasoning over~$\RCA_0 + \BSig_2 + T$ for a $\Pi^0_3$ sentence~$T$.

For the remainder of this article, and unless specified, the arithmetic hierarchy is extended to allow the use of the constants $\Param$ and $\param$ inside formulas. Let $\theta(x, y, z)$ be the $\Sigma^0_0$-formula $\neg \zeta(\Param \uh_z,\param,x,y,z)$ and let~$T \equiv \forall x \exists y \forall z \theta(x, y, z)$.

\begin{definition}
Two finite sets~$X < Y$ are \emph{$T$-apart} if 
$$\forall x < \max X \exists y < \min Y \forall z < \max Y \theta(x, y, z)$$
\end{definition}

Note that $T$-apartness is a transitive relation. Moreover, if $X < Y$ are $T$-apart and $X_0 \subseteq X$ and $Y_0 \subseteq Y$, then $X_0, Y_0$ are $T$-apart.

\begin{definition}[Largeness for $\RCA_0 + \BSig_2 + T$]\label[definition]{defi:largeness-rca0-bsig2-t}
A set~$X \finsub \NN$ is 
\begin{itemize}
    \item \emph{$\bbomega^0$-large$(T)$} if $X \neq \emptyset$
    \item \emph{$\bbomega^{(n+1)}$-large$(T)$} if $X \setminus \min X$ is  $(\bbomega^n \cdot \min X)$-large$(T)$
    \item \emph{$\bbomega^n \cdot k$-large$(T)$} if
there are $k$ pairwise $T$-apart $\bbomega^n$-large$(T)$ subsets of~$X$
$$
X_0 < X_1 < \dots < X_{k-1}
$$
\end{itemize}
\end{definition}

As for $\bbomega^n$-largeness, we will only consider sets $X$ such that $\min X \geq 3$. In addition, we will also require that $\min X \geq \param$ to ensure that the constant $\param$ will always be in the new model obtained by building a proper cut (see \Cref{prop:initial-segment-wkl-bsig2-t-largeness}).

Note that if we take $\theta(x, y, z)$ to be the $\top$ formula, then $\bbomega^n \cdot k$-largeness$(T)$ is exactly $\bbomega^n \cdot k$-largeness.
We first prove some basic closure properties for~$\bbomega^n$-largeness$(T)$. A stronger closure property will be proven in~\Cref{sect:largeness-closure}.

\begin{lemma}\label[lemma]{lem:largeness-closure-succ}
$\RCA_0 + \BSig_2 + T$ proves that for every~$b \in \NN$, if $\bbomega^b$-largeness$(T)$ is a largeness notion, then
\begin{enumerate}
    \item For every~$c \in \NN$, $\bbomega^b \cdot c$-largeness$(T)$ is a largeness notion ;
    \item $\bbomega^{b+1}$-largeness$(T)$ is a largeness notion.
\end{enumerate}
\end{lemma}
\begin{proof}
Suppose that $\bbomega^b$-largeness$(T)$ is a largeness notion.

Fix an infinite set~$X$. We prove by internal $\Sigma^0_1(X)$ induction over~$c \geq 1$, that $X$ contains an 
$\bbomega^b \cdot c$-large$(T)$ subset. The case~$c = 1$ follows from the fact that $\bbomega^b$-largeness$(T)$ is a largeness notion. Assume the case $k$ holds. Let~$F \subseteq X$ be an $\bbomega^b \cdot k$-large$(T)$ set.
Since~$T$ holds, then by $\BSig_2$, there is some~$d \in X$ such that 
$$
\forall x < \max F \exists y < d \forall z \theta(x, y, z)
$$
Since~$X \setminus [0, d]$ is infinite and $\bbomega^b$-largeness$(T)$ is a largeness notion, 
there is an $\bbomega^b$-large$(T)$ subset~$G \subseteq X \setminus [0, d]$.
Then $F \cup G$ is an $\bbomega^b \cdot (c+1)$-large$(T)$ subset of~$X$. 
Since every infinite set~$X$ contains a $\bbomega^b \cdot c$-large$(T)$ subset, then $\bbomega^b \cdot c$-largeness$(T)$ is a largeness notion.

Suppose now that for every~$c \geq 1$, $\bbomega^b \cdot c$-largeness$(T)$ is a largeness notion. In particular, for every infinite set~$X$, there is an $\bbomega^b \cdot (\min X)$-large$(T)$ set~$F \subseteq X \setminus \{\min X\}$, hence $\{\min X\} \cup F$ is an $\bbomega^{b+1}$-large$(T)$ subset of~$X$. Thus, $\bbomega^{b+1}$-largeness$(T)$ is a largeness notion.
\end{proof}

\begin{proposition}\label[proposition]{prop:largeness-bsig2-largeness}
For every~$n \in \omega$,
$\RCA_0 + \BSig_2 + T$ proves that for every~$c \geq 1$, $\bbomega^n \cdot c$-largeness$(T)$ is a largeness notion.
\end{proposition}

\begin{proof}
We prove by external induction over~$n$ that $\bbomega^n$-largeness$(T)$ is a largeness notion. Then statement of the proposition then follows from~\Cref{lem:largeness-closure-succ}.
For the base case, since every non-empty finite set $X$ is $\bbomega^0$-large$(T)$, then $\bbomega^0$-largeness$(T)$ is a largeness notion.
The induction case is the second item of \Cref{lem:largeness-closure-succ}.

\end{proof}

A set~$X$ is \emph{exp-sparse} if for every~$x < y \in X$, $4^x < y$.

\begin{corollary}\label[corollary]{prop:largeness-bsig2-largeness-sparsity}
For every~$n \in \omega$,
$\RCA_0 + \BSig_2 + T$ proves that for every~$c \geq 1$, exp-sparse $\bbomega^n \cdot c$-largeness$(T)$ is a largeness notion.
\end{corollary}
\begin{proof}
Fix~$n \in \omega$ and $c \geq 1$. Let~$X$ be an infinite set.
Then, there is an infinite exp-sparse subset~$Y \subseteq X$.
By \Cref{prop:largeness-bsig2-largeness}, there is an $\bbomega^n \cdot c$-large$(T)$ subset~$F \subseteq Y$. In particular, $F$ is an exp-sparse $\bbomega^n \cdot c$-large$(T)$ subset of~$X$.
\end{proof}

\subsection{Largeness$(T)$ combinatorics}

The following lemma is inspired from Ko\l odziejczyk and Yokoyama~\cite[Lemma 2.2]{kolo2020some},
but required some variations of the combinatorics to handle the constraints based on~$T$.

\begin{lemma}\label[lemma]{lem:largeness-rt1k}
$\ISig_1$ proves that for all $b \in \NN$, every $\bbomega^{2b}$-large$(T)$ exp-sparse set~$X$ and every coloring~$f : X \to \min X$, there is an $\bbomega^b$-large$(T)$ $f$-homogeneous subset~$Y \subseteq X$.
\end{lemma}
\begin{proof}
By induction over~$b$.

Case $b = 0$. For every $\bbomega^0$-large$(T)$ set $X$ and every~$f : X \to \min X$,
every 1-element subset of~$X$ is $f$-homogeneous and $\bbomega^0$-large$(T)$.

Case $b > 0$. Fix an $\bbomega^{2b}$-large$(T)$ exp-sparse set~$X$ and $f : X \to \min X$.
Since~$X$ is $\bbomega^{2b}$-large$(T)$, there are $\min X$ pairwise $T$-apart $\bbomega^{2b-1}$-large$(T)$ subsets $X_0 < \dots < X_{\min X - 1}$ of $X \setminus \min X$.
Suppose $X_0$ is not $f$-homogeneous, otherwise we are done.
Let~$t < \min X$ be maximal such that $f[\bigcup_{i < t} X_i] \supseteq f[X_t]$.
Note that $t \geq 1$ exists since otherwise, 
$$2 \leq |f[X_0]| < |f[X_0 \cup X_1]| < \dots < |f[X_0 \cup \dots \cup X_{\min X-1}]|$$
and therefore $|f[X]| \geq \min X+1$, contradicting the assumption that $f : X \to \min X$.

Since~$X_t$ is $\bbomega^{2b-1}$-large$(T)$, there are $\min X_t$ pairwise $T$-apart $\bbomega^{2b-2}$-large$(T)$ subsets $Y_0 < \dots < Y_{\min X_t - 1}$ of $X_t \setminus \min X_t$.

By induction hypothesis, for every~$j < \min X_t$,
there is a $\bbomega^{b-1}$-large$(T)$ subset $Z_j \subseteq Y_j$ which is $f$-homogeneous for some color~$c_j < \min X$.
Since $X$ is exp-sparse, $4^{\max X_{t-1}} < \min X_t$ so $(\max X_{t-1})^2 < \min X_t$, thus $\min X \times \max X_{t-1} < \min X_t$. Thus, by the finite pigeonhole principle, there is a subset $J \subseteq \{0, \dots, \min X_t - 1\}$ of size at least~$\max X_{t-1}$ and some color~$c < \min X$ such that for every~$j \in J$, $c_j = c$.
By choice of $t$, $f[\bigcup_{i < t} X_i] \supseteq f[X_t]$, so there is some~$x \in \bigcup_{i < t} X_i$ with $f(x) = c$. Thus, the set~$\{x\} \cup \bigcup_{j \in J} Z_j$ is an $\bbomega^b$-large$(T)$ $f$-homogeneous subset of~$X$. 
\end{proof}

\begin{lemma}\label[lemma]{lem:decompose-largeness}
Let~$n, m \in \omega$. $\ISig_1$ proves that
for every~$\bbomega^{n+m+1}$-large$(T)$ set~$X$,
there are $\bbomega^n$-large$(T)$ pairwise $T$-apart subsets $X_0 < \dots < X_{k-1}$ of~$X$
such that $\{ \min X_i : i < k \}$ is $\bbomega^m$-large (in the regular sense of largeness)
\end{lemma}
\begin{proof}
By external induction over $m$, we prove the following statement that directly imply the lemma (assuming that $\min X \geq 3$): for all $T$-apart pairs $Y_0 < Y_1$ of $\bbomega^{n+m}$-large$(T)$ sets, there are $\bbomega^n$-large$(T)$ subsets $X_1 < \dots < X_{k-1}$ of $Y_1$ such that, letting $X_0 = Y_0$, $\{ \min X_i : i < k \} $ is $\bbomega^m$-large and $X_0, \dots, X_{k-1}$ are pairwise $T$-apart. 

Case $m = 0$. The result is clear, as every non-empty set with an element greater than $3$ is $\bbomega^0$-large$(T)$

Case $m > 0$. Let $Y_0 < Y_1$ be $\bbomega^{n+m}$-large$(T)$ and $T$-apart, let $Z_0 < \dots < Z_{\min Y_1 - 1}$ be $\bbomega^{n+m-1}$-large$(T)$ pairwise $T$-apart subsets of $Y_1$. Since $Y_0$ is $\bbomega^{n+m}$-large$(T)$, then it is $\bbomega$-large and therefore $\min Y_1 > \max Y_0 \geq 2\times (\min Y_0)$.

We can then apply the inductive hypothesis on the pairs $$(Z_0,Z_1), \dots, (Z_{2(\min Y_0) - 2}, Z_{2(\min Y_0) - 1})$$
to get for every $j < \min Y_0$, families of pairwise $T$-apart $\bbomega^{n}$-large$(T)$ subsets $Z_{2j} = X_{j,0} < \dots < X_{j,k_j}$ of $Z_{2j} \cup Z_{2j+1}$ such that $\{\min X_{j,i} : i < k_j\}$ is $\bbomega^{m-1}$-large.

Consider the family, $Y_0 < X_{0,0} < X_{0,1} < \dots < X_{0,k_0} < X_{1,0} < X_{1,1} < \dots < X_{\min Y_0 - 1, k_{\min Y_0 - 1}}$. Every block of this family is $\bbomega^n$-large$(T)$. Since $X_{0,0} \subseteq Y_1$, then $Y_0$ and $X_{0,0}$ are $T$-apart. Moreover, for all $j < \min Y_0 - 1$, since $X_{j,k_j} \subseteq Z_{2j+1}$ and $X_{j+1,0} = Z_{2j+2}$, $X_{j,k_j}$ and $X_{j+1,0}$ are $T$-apart.
So by denoting $W_i$ the $i$-th element of the family and by $k$ the cardinality of the family, $\{\min W_i | i < k\}$ is $\bbomega^m$-large and the $W_i$'s are pairwise $T$-apart.
This completes the proof.
\end{proof}


\subsection{Closure of largeness$(T)$}\label[section]{sect:largeness-closure}

Given an ordinal~$\alpha \leq \epsilon_0$, we let $\WF(\alpha)$ be the $\forall\Pi^0_3$ statement of the well-foundedness of~$\alpha$. It is well-known that, over~$\RCA_0$, $\WF(\alpha)$ is equivalent to the fact that every infinite set admits an $\alpha$-large subset (see \cite[Lemma 3.2]{kolo2020some}). This motivates the following definition:

\begin{statement}
$\WF_T(\bbomega^b)$: Every infinite set admits an $\bbomega^b$-large$(T)$ subset. In other words, $\bbomega^b$-largeness$(T)$ is a largeness notion.    
\end{statement}

For every $n \in \omega$, $\RCA_0$ proves that $\WF(\bbomega^n)$ holds, but $\RCA_0$ does not prove~$\WF(\bbomega^\bbomega)$. Thus, given a model~$\M = (M, S)$, the set $\WF(\bbomega^\M) = \{ b \in M : \M \models \WF(\bbomega^b) \}$ is a cut of~$\M$ which is proper iff $\M \not \models \WF(\bbomega^\bbomega)$. $\RCA_0$ proves that $\WF(\M)$ is an additive cut, that is, for every~$b \in \NN$, $\WF(\bbomega^b)$ implies $\WF(\bbomega^{2b})$ (see \cite[Lemma 3.2]{kolo2020some}). We prove the same property for~$\WF_T(\bbomega^\M) = \{ b \in M : \M \models \WF_T(\bbomega^b) \}$.

\begin{lemma}[$\RCA_0 + \BSig_2 + T$]\label[lemma]{lem:large-set-of-large-sets-is-large}
Let~$a, b \in \NN$, $k \geq 0$ and $X_0 < X_1 < \dots < X_{k-1}$ be pairwise $T$-apart $\bbomega^a$-large$(T)$ sets such that $\{\max X_s : s < k \}$ is $\bbomega^{b+1}$-large$(T)$. 
Then $\{\max X_0\} \cup \bigcup_{1 \leq s < k} X_s$ is $\bbomega^{a+b}$-large$(T)$.
\end{lemma}
\begin{proof}
By induction over~$b$. If~$b = 0$, then since any $\bbomega^1$-large$(T)$ set has cardinality at least~2, $k > 1$, and by assumption, $X_1$ is $\bbomega^a$-large$(T)$, so $\{\max X_0\} \cup \bigcup_{1 \leq s < k} X_s$ is $\bbomega^{a+b}$-large$(T)$.

Let~$b > 0$. Let~$Z_0 < Z_1 < \dots Z_{\ell-1}$ be pairwise $T$-apart $\bbomega^b$-large$(T)$ sets such that $\{\ell\} \cup Z_0 \cup \dots \cup Z_{\ell-1}$ is an $\bbomega^{b+1}$-large$(T)$ subset of~$\{\max X_s : s < k \}$. By induction hypothesis, for every~$t < \ell$, the set $W_t = \{\min Z_t\} \cup \bigcup \{ X_s : \max X_s \in Z_t \setminus \min Z_t\}$ is $\bbomega^{a+b-1}$-large$(T)$. Note that $\min Z_t = \min W_t$ and $\max W_t = \max Z_t$. Given~$i < j < \ell$, since~$Z_i$ and $Z_j$ are $T$-apart, and since $\max Z_i = \max W_i$, $\min Z_j = \min W_j$ and $\max W_j = \max Z_j$, then $W_i$ and $W_j$ are $T$-apart. Last, since~$\max X_0 \leq \ell$, then $\{\max X_0\} \cup W_0 \cup \dots \cup W_{\ell-1}$ is an $\bbomega^{a+b}$-large$(T)$ subset of $\{\max X_0\} \cup \bigcup_{1 \leq s < k} X_s$.
\end{proof}

\begin{proposition}
$\RCA_0 + \BSig_2 + T$ proves that for every~$a \in \NN$, if $\bbomega^a$-largeness$(T)$ is a largeness notion, then so is $\bbomega^{2a}$-largeness$(T)$.
\end{proposition}
\begin{proof}
Suppose $\bbomega^a$-largeness$(T)$ is a largeness notion. 
Fix an infinite set~$X$. We build greedily an infinite subsequence $X_0 < X_1 < \dots$ of pairwise $T$-apart $\bbomega^a$-large$(T)$ sets. Suppose for the contradiction that there is some~$k \in \NN$ such that $X_{k+1}$ is not found.

Since~$T$ holds, then by $\BSig_2$, there is some~$b \in X$ such that 
$$
\forall x < \max X_k \exists y < b \forall z \theta(x, y, z)
$$
Since~$X \setminus [0, b]$ is infinite and $\bbomega^n$-largeness$(T)$ is a largeness notion, 
there is an $\bbomega^n$-large$(T)$ subset~$X_{k+1} \subseteq X \setminus [0, b]$.
Then $X_0, \dots, X_{k+1}$ are pairwise $T$-apart, contradiction.

By \Cref{lem:largeness-closure-succ}, $\bbomega^{a+1}$-largeness$(T)$ is a largeness notion.
Since~$Y = \{\max Y_s : s \in \NN \}$ is an infinite set, then there is an $\bbomega^{a+1}$-large$(T)$ subset~$F$.
By \Cref{lem:large-set-of-large-sets-is-large}, $\bigcup \{ X_s : \max X_s \in F \}$ is an $\bbomega^{2a}$-large$(T)$ subset of~$X$.
\end{proof}

\subsection{A generalized Parsons theorem for $\WKL_0 + \BSig_2 + T$}

We now turn to the proof of a generalized Parsons theorem for this notion of largeness. The proof is an elaboration of the original construction by Kirby and Paris~\cite{Kirby1977InitialSO}. It is based on the construction of a semi-regular cut which satisfies some extra properties to make it a model of~$\RCA_0 + \BSig_2 + T$. We start with a few definitions:

\begin{definition}[Cut]
A \emph{cut} in a model $M$ of first order arithmetic is a nonempty subset $I \subseteq M$ which is closed by successor, and is an initial segment of~$M$, that is, if $a \in I$ and $b \leq a$, then $b \in I$

\end{definition}


\begin{definition}[Semi-regular cut \cite{Kirby1977InitialSO}]
A cut $I \subseteq M$ is said to be \emph{semi-regular} if for every $M$-finite set $E \subseteq M$ such that $|E| \in I$, $E \cap I$ is bounded in $I$.
\end{definition}

The following proposition shows that semi-regular cuts are exactly those coding a model of~$\WKL_0$:

\begin{proposition}[{See Scott \cite{scott1962algebras} and Kirby and Paris \cite[Proposition 1]{Kirby1977InitialSO}}]
Let $I \subseteq M$ be a cut and let $\mathsf{Cod}(M/I) = \{S \cap I : S \mbox{ is } M\mbox{-finite}\}$. Then $I$ is semi-regular if and only if $(I, \mathsf{Cod}(M/I)) \models \WKL_0$.   
\end{proposition}

\begin{proposition}\label[proposition]{prop:initial-segment-wkl-bsig2-t-largeness}
Given a countable non-standard model $\M = (M,S,\param^{\M},\Param^{\M})$ of $\ISig_1$ and an $M$-finite set~$Z \subseteq M$ which is $\bbomega^{2^c}$-large$(T)$ and exp-sparse for some~$c \in M \setminus \omega$, there is an initial segment~$I$ of $M$ such that $(I, \mathsf{Cod}(M/I),\param^{\M},\Param^{\M}\cap I) \models \WKL_0 + \BSig_2 + T$ and $I \cap Z$ is infinite in~$I$.
\end{proposition}

\begin{proof}
Let $(E_i)_{i \in \omega}$ be an enumeration containing all the $M$-finite sets infinitely many times each and let $(f_i)_{i \in \omega}$ be an enumeration containing infinitely many times all the $M$-finite functions from $\{0, \dots, \max Z\}$ to a $k$ where $k < \max Z$.

We will build a decreasing sequence of sets $Z = Z_0 \supseteq Z_1 \supseteq \dots$ such that $Z_i$ will be $\bbomega^{2^{c - i}}$-large$(T)$ (and still exp-sparse). \\

At stages of the form $s = 4i$, let $Z_{4i}$ be given. If $\min Z_{4i} \leq |E_i|$, then keep $Z_{4i+1} = Z_{4i}$. If $\min Z_{4i} > |E_i|$, then, letting $Z_{4i}^0 < \dots < Z_{4i}^{\min Z_{4i} -1}$ be $\bbomega^{2^{c - 4i - 1}}$-large$(T)$ such that $\bigsqcup_{j < \min Z_{4i}} Z_{4i}^j \subseteq Z_{4i}$, by the finite pigeonhole principle, there exists a $j < \min Z_{4i}$ such that $E_i \cap Z_{4i}^j = \emptyset$, in this case, take $Z_{4i+1} = Z_{4i}^j$ for such a $j$.

At stages of the form $s = 4i + 1$, let $Z_{4i+1}$ be given. Let $k$ be the range of $f_i$, if $\min Z_{4i+1} < k$, then keep $Z_{4i+2} = Z_{4i+1}$. If $\min Z_{4i+1} > k$ then $f_i$ induce a coloring $\tilde{f_i} : Z_{4i+1} \to \min Z_{4i+1}$ and let $Z_{4i+2}$ be an $\bbomega^{2^{c - 4i - 2}}$-large$(T)$ $\tilde{f_i}$-homogeneous subset of $Z_{4i+1}$ (by \Cref{lem:largeness-rt1k}).

At stages of the form $s = 4i+2$, let $Z_{4i+2}$ be given and put $Z_{4i+3} = Z_{4i+2} \setminus \{\min Z_{4i+2}\}$. 

At stages of the form $s = 4i +3$, let $Z_{4i+3}$ be given. $Z_{4i+3}$ contains at least two $\bbomega^{2^{c - 4i - 4}}$-large$(T)$ and $T$-apart subsets, $Y_0 < Y_1$. Take $Z_{4i+4} = Y_1$.


Finally, let $I = \sup\{\min Z_i | i \in \omega\}$. First, note that by convention, for every $\bbomega^{2^c}$-large$(T)$ set~$Z$, $\param^{\M} \leq \min X$, hence $\param^\M \in I$.

The stages of the form $s = 4i$ ensure that $I$ is a semi-regular cut and therefore $(I, \mathsf{Cod}(M/I)) \models \WKL_0$. 

The stages of the form $s = 4i+2$ ensure that $Z_i \cap I$ is infinite in $I$ for every $i$ (and in particular $Z \cap I$ is infinite in $I$). 

The stages of the form $s = 4i+1$ ensure that $(I, \mathsf{Cod}(M/I)) \models \RT^1$ (and therefore $\BSig_2$): let $f : I \to k \in \mathsf{Cod}(M/I)$, there exists an index $i \in \omega$ such that $f = f_i \cap I$ (and therefore $f = f_i \uh I$) and since $k \in I$, we can take the $i$ to be large enough for $\min Z_{4i+1}$ to be bigger than $k$ (Since every such function appears infinitely many times in the enumeration). By construction $Z_{4i+2}$ is $f_i$-homogeneous, so $Z_{4i+2} \cap I$ is an element of $\mathsf{Cod}(M/I)$ that is $f$-homogeneous and infinite in $I$.

The stages of the form $s = 4i +3$ ensure that $(I, \mathsf{Cod}(M/I),\param^{\M},\Param^{\M} \cap I) \models T$ as for every $k \in I$, there exists an index $i$ such that $k < \min Z_{4i + 3}$ and therefore $\forall x < k \exists y < \min Z_{4i+4} \forall z < \max Z_{4i+4} \ \theta(x,y,z)$ holds. Replacing $\Param^{\M}$ by $\Param^{\M} \cap I$ inside $\theta(x,y,z)$ does not change its truth value for $x,y,z \in I$, hence $(I, \mathsf{Cod}(M/I),\param^{\M},\Param^{\M}\cap I) \models \forall x < k \exists y \forall z \ \theta(x,y,z)$. 
\end{proof}

\begin{reptheorem}{thm:generalized-parsons-p4-bsig2-largeness}[Generalized Parsons theorem for $\WKL_0 + \BSig_2 + T$]
Let~$\psi(x, y, F)$ be a $\Sigma_0$ formula with exactly the displayed free variables. Assume that 
$$
\WKL_0 + \BSig_2 + T \vdash \forall x \forall X(X \mbox{ is infinite } \rightarrow \exists F \finsub X \exists y \psi(x, y, F))
$$
Then there exists some~$n \in \omega$ such that $\ISig_1$ proves $\forall x \forall Z \finsub (x, \infty)$
$$
Z \mbox{ is } \bbomega^n\mbox{-large}(T) \mbox{ and exp-sparse } \rightarrow \exists F \subseteq Z \exists y < \max Z \psi(x, y, F)
$$
\end{reptheorem}

\begin{proof}
By contradiction, assume that for all $n \in \omega$, $\ISig_1$ does not prove 
$$\forall x \forall Z \finsub (x, \infty)[Z \mbox{ is } \bbomega^n\mbox{-large}(T) \mbox{ and exp-sparse } \rightarrow \exists F \subseteq Z \exists y < \max Z \psi(x, y, F)]
$$

There is a countable model $\M = (M,S) \models \ISig_1$, satisfying for every~$n \in \omega$
$$
\exists x \exists Z \finsub (x, \infty)\left[\begin{array}{l}Z \mbox{ is } \bbomega^{2^n}\mbox{-large}(T) \mbox{ and exp-sparse }\\ \mbox{and } \forall F \subseteq Z \forall y < \max Z \neg \psi(x, y, F)\end{array}\right]
$$

We can assume $M$ to be non-standard, so by overspill, there exists a $b \in M \setminus \omega$ such that 
$$M \models \exists x \exists Z \finsub (x, \infty)\left[\begin{array}{l}Z \mbox{ is } \bbomega^{2^b}\mbox{-large}(T) \mbox{ and exp-sparse }\\ \mbox{and } \forall F \subseteq Z \forall y < \max Z \neg \psi(x, y, F)\end{array}\right]$$

Let $c \in M$ and $Z \finsub (c, \infty)$ $\bbomega^{2^b}$-large$(T)$ be such that $\M \models \forall F \subseteq Z \forall y < \max Z \neg \psi(c, y, F)$.
By \Cref{prop:initial-segment-wkl-bsig2-t-largeness}, there is an initial segment $I$ of $M$ such that $(I, \mathsf{Cod}(M/I),\param^{\M},\Param^{\M} \cap I) \models \WKL_0 + \BSig_2 + T$ and $Z \cap I$ infinite in $I$. Therefore, $$(I, \mathsf{Cod}(M/I),\param^{\M},\Param^{\M} \cap I) \models (Z \cap I \mbox{ is infinite } \wedge \forall F \finsub Z \cap I \forall y \neg \psi(c,y,F))$$
This contradicts our assumption that
$$
\WKL_0 + \BSig_2 + T \vdash \forall x \forall X(X \mbox{ is infinite } \rightarrow \exists F \finsub X \exists y \psi(x, y, F))$$    
\end{proof}

\section{A framework for $\forall \Pi^0_4$ conservation}\label[section]{sect:framework-pi04}







We now develop a framework for proving that $\RT$-like statements are $\forall \Pi^0_4$ conservative over~$\RCA_0 + \BSig_2$. This will be divided into two sections : First, building up on the work of Kirby and Paris~\cite{Kirby1977InitialSO}, we prove in \Cref{sect:density-ramsey-like} that every $\RT$-like theorem~$\Gamma$ is $\forall \Pi^0_4$ conservative over~$\RCA_0 + \BSig_2 + $ a first-order theory stating the existence of dense$(\Gamma)$ sets. Then, in \Cref{sect:largeness-ramsey-like}, we use the quantitative notion of largeness defined in \Cref{sect:largeness-t} to prove the existence of dense$(\Gamma)$ sets, assuming the existence of a combinatorially simpler objects.

Given a function $f : [\NN]^n \to k$ and a finite set~$G = \{x_0 < \dots < x_{\ell-1}\} $, let $f_G : [\ell]^n \to k$ be defined by $f_G(i_0, \dots, i_{n-1}) = f(x_{i_0}, \dots, x_{i_{n-1}})$.

\begin{definition}[$\RT$-like formula]\label[definition]{def:ramsey-like-Pi12-formula}
Given~$n, k \in \omega$, an \emph{$\RT$-like formula} is a $\Pi_2^1$-formula of the form 
$$(\forall f : [\mathbb{N}]^n \to k)(\exists Y)(Y \mbox{is infinite} \wedge \Psi(f,Y))$$
where $\Psi(f,Y)$ is of the form $(\forall G \finsub Y) \Psi_0(f_G)$ with $\Psi_0$ a $\Delta^0_0$-formula.
\end{definition}

$\RT$-like statements were introduced and studied in the standard realm by Patey~\cite{patey2022ramseylike}, formulated in terms of forbidden patterns. The idea behind the formula $\Psi$ is that theorems from Ramsey theory are about structure and not the actual value of the domain. Thus, every solution is rewritten as a copy of the integers.

\begin{remark}
The notion of $\RT$-like statement is a particular case of Ramsey-like-$\Pi_2^1$-formula, introduced by defined by Patey and Yokoyama~\cite{patey2018proof}. It already encompasses Ramsey's theorem for pairs, the Ascending Descending Sequence principle and the Erd\H{o}s-Moser theorem. This section could have been formulated and proven in the more general setting of Ramsey-like-$\Pi_2^1$-formulas, with the cost of some heavier notations as in~\cite{patey2018proof}.
\end{remark}

Since an $\RT$-like statement does not distinguish the solutions from the set of the integers, it proves a slightly stronger formula, stating that the class of solutions is dense in the partial order $([\NN]^\NN, \subseteq)$. 

\begin{lemma}\label[lemma]{lem:rtnk-like-density}
Let $\Gamma \equiv (\forall f : [\mathbb{N}]^n \to k)(\exists Y)(Y \mbox{is infinite} \wedge \Psi(f,Y))$ be an $\RT$-like formula.
Then $\RCA_0 + \Gamma$ proves
$$(\forall f : [\mathbb{N}]^n \to k)(\forall X \mbox{ infinite})(\exists Y \subseteq X)(Y \mbox{is infinite} \wedge \Psi(f,Y))$$
\end{lemma}
\begin{proof}
Let $f : [\NN]^n \to k$ be a coloring and $X = \{ x_0 < x_1 < \dots \}$ be an infinite set.
Let $\hat{f} : [\NN]^n \to k$ be defined by $\hat{f}(i_0, \dots, i_{n-1}) = f(x_{i_0}, \dots, x_{i_{n-1}})$.
By $\Gamma$, there is an infinite set $\hat{Y} \subseteq \NN$ such that $\Psi(\hat{f}, \hat{Y})$.
Let~$Y = \{ x_s : s \in \hat{Y} \}$. We claim that $\Psi(f, Y)$ holds.

Say $\Psi(f, Y) \equiv (\forall G \finsub Y) \Psi_0(f_G)$ with $\Psi_0$ a $\Delta^0_0$-formula.
Let $G \finsub Y$, and let $\hat{G} = \{ s : x_s \in G \}$. 
Say $G = \{x_{s_0}, \dots, x_{s_{\ell-1}} \}$. Then 
$$f_G(i_0, \dots, i_{n-1}) = f(x_{s_{i_0}}, \dots, x_{s_{i_{n-1}}}) = \hat{f}(s_{i_0}, \dots, s_{i_{n-1}}) = \hat{f}_{\hat{G}}(i_0, \dots, i_{n-1})$$
hence $f_G = \hat{f}_{\hat{G}}$. Since $\Psi(\hat{f}, \hat{Y})$ and $\hat{G} \finsub \hat{Y}$, then $\Psi_0(\hat{f}_{\hat{G}})$ holds, then so does $\Psi_0(f_G)$. Since it holds for every~$G \finsub Y$, then $\Psi(f, Y)$ holds.
\end{proof}


\subsection{Density$(T)$ for $\RT$-like statements}\label[section]{sect:density-ramsey-like}

There exists a very general way to characterize the first-order part of a second-order theory, using the notion of density, originally defined by Kirby and Paris~\cite{Kirby1977InitialSO}. This was later adapted by Bovykin and Weiermann~\cite[Theorem 1]{bovykin2017strength} to prove $\Pi^0_2$ conservation of combinatorial theorems, and generalized by Patey and Yokoyama~\cite[Theorem 3.4]{patey2018proof} for~$\forall \Pi^0_3$ conservation.

\begin{definition}\label[definition]{def:density}
Fix an $\RT$-like statement $$\Gamma = (\forall f : [\mathbb{N}]^n \to k)(\exists Y)(Y \mbox{is infinite} \wedge \Psi(f,Y)).$$

We define inductively the notion of $m\mbox{-density}(T, \Gamma)$ of a finite set $Z \subseteq \mathbb{N}$ as follows. A set $Z$ is $0\mbox{-dense}(T, \Gamma)$ if it is $\bbomega\mbox{-large}(T)$ and a set $Z$ is $(m+1)\mbox{-dense}(T, \Gamma)$ if

\begin{itemize}
    \item[(a)] for any $f : [Z]^n \to k$, there is an $m\mbox{-dense}(T, \Gamma)$ set $Y \subseteq Z$ such that $\Psi(f,Y)$ holds, and,
    \item[(b)] for any partition $Z_0 \sqcup \dots \sqcup Z_{\ell-1} = Z$ such that $\ell \leq Z_0 < \dots < Z_{\ell-1}$, one of the $Z_i$'s is $m\mbox{-dense}(T, \Gamma)$. 
    \item[(c)] for any $f : Z \to \min Z$, there is an $m\mbox{-dense}(T, \Gamma)$ set $Y \subseteq Z$ which is $f$-homogeneous
    \item[(d)] there is an $m\mbox{-dense}(T, \Gamma)$ set $Y \subseteq Z$ such that
    $$\forall x < \min Z \exists y < \min Y \forall z < \max Y\ \theta(x, y, z)$$
\end{itemize}  
\end{definition}

As for largeness$(T)$, we require that any $m$-dense$(T, \Gamma)$ set $X$ satisfies $\min X \geq \max \{3,\param\}$.

\begin{remark}\label[remark]{rem:density}\ 
\begin{itemize}
    \item[(1)] In the definition above, item (a) is an indicator for $\Gamma$, item (b) for $\WKL_0$, item (c) for~$\RT^1$ and item (d) for~$T$. Strictly speaking, this notion corresponds to $m\mbox{-density}(\Gamma, \WKL_0, \RT^1, T)$, but for simplicity of notation, we only made explicit the varying parameters.
    \item[(2)] Note that item (b) follows from item (c). However, for clarity and explicitness of the indicators, we kept both items.
    \item[(3)] Also note that there is a hidden use of $\BSig_2$ in item (d).
Indeed, one assume that every~$x < \min Z$ will have $\min Y$ as a uniform bound.
One could have modified this item to require that for every~$x < \min Z$, there is an $m\mbox{-dense}(T,  \Gamma)$ set $Y \subseteq Z$ and some~$y < \min Y$
such that $\forall z < \max Z\ \theta(x, y, z)$.
\end{itemize}
\end{remark}

We now prove the core combinatorial lemma which relates the notion of density to the existence of a cut satisfying the desired properties.

\begin{lemma}\label[lemma]{prop:initial-segment-wkl-bsig2-t-gamma-density}
Let $\Gamma$ be an $\RT$-like statement. Given a countable nonstandard model $\M \models \ISig_1$ and an $M$-finite set $Z \subseteq M$ which is $c\mbox{-dense}(T, \Gamma)$ for some $c \in M \setminus \omega$, then there exists an initial segment $I$ of $M$ such that $(I, \mathsf{Cod}(M/I),\param^{\M},\Param^{\M} \cap I) \models \WKL_0 + \BSig_2 + \Gamma + T$ and $I \cap Z$ is infinite in $I$. 
\end{lemma}

\begin{proof}
Let $\Gamma$ be an $\RT$-like statement of the form
$$(\forall f : [\mathbb{N}]^n \to k)(\exists Y)(Y \mbox{is infinite} \wedge \Psi(f,Y))$$

where $k,n \in \omega$ and $\Psi$ in the form of \Cref{def:ramsey-like-Pi12-formula}

Let $(E_i)_{i \in \omega}$ be an enumeration containing all the $M$-finite sets infinitely many times each, let $(f_i)_{i \in \omega}$ be an enumeration of all the $M$-finite functions from $[\{0, \dots, \max Z\}]^n$ to $k$ and let $(g_i)_{i \in \omega}$ be an enumeration containing infinitely many times all the $M$-finite functions from $\{0, \dots, \max Z\}$ to a $l$ where $l < \max Z$.

We will build a decreasing sequence of sets $Z = Z_0 \supseteq Z_1 \supseteq \dots$ such that $Z_i$ will be $(c-i)\mbox{-dense}(T, \Gamma)$. \\

At stages of the form $s = 5i$, let $Z_{5i}$ be given. If $\min Z_{5i} \leq |E_i|$, then keep $Z_{5i+1} = Z_{5i}$. If $\min Z_{5i} > |E_i|$, let $E_i = \{e_0, \dots, e_{l-1}\}$ where $e_0 < \dots < e_{l-1}$ and put $Z_{5i}^0 = Z_{5i} \cap [0, e_0)_{M}$, $Z_{5i}^j = Z_{5i} \cap [e_{j-1}, e_{j})_{M}$ for $1 \leq j < l - 1$ and $Z_{5i}^l = Z_{5i} \cap [e_{l-1}, \infty)_{M}$. Then $Z_{5_i} = Z_{5i}^0 \sqcup \dots \sqcup Z_{5i}^l$, thus one of the $Z_{5i}^j$ for $j \leq l$ is $(c - 5i - 1)\mbox{-dense}(T, \Gamma)$. Put $Z_{5i+1}$ to be such a $Z_{5i}^j$.

At stages of the form $s = 5i + 1$, let $Z_{5i+1}$ be given. Let $l$ be the range of $g_i$, if $\min Z_{5i+1} < l$, then keep $Z_{5i+2} = Z_{5i+1}$. If $\min Z_{5i+1} > l$ then $g_i$ induce a coloring $\tilde{g_i} : Z_{5i+1} \to \min Z_{5i+1}$ and let $Z_{5i+2}$ be an $(c- 5i-2)\mbox{-dense}(T, \Gamma)$ $\tilde{g_i}$-homogeneous subset of $Z_{5i+1}$. 

At stages of the form $s = 5i+2$, let $Z_{5i+2}$ be given and put $Z_{5i+3} = Z_{5i+2} \setminus \{\min Z_{5i+2}\}$. 

At stages of the form $s = 5i +3$, let $Z_{5i+3}$ be given and let $Z_{5i+4}$ be a $(c- 5i-4)\mbox{-dense}(T, \Gamma)$ subset of $Z_{5i+3}$ such that
$$\forall x < \min Z_{5i+3} \exists y < \min Z_{5i+4} \forall z < \max Z_{5i+4} \ \theta(x, y, z)$$

At stages if the form $s = 5i + 4$, let $Z_{5i+4}$ and $f_i$ be given. Let $Z_{5i+5} \subseteq Z_{5i+4}$ satisfying $\Psi(f_i,Z_{5i+5})$. \\

Finally, let $I = \sup\{\min Z_i | i \in \omega\}$. \\

The stages of the form $s = 5i$ ensure that $I$ is a semi-regular cut and therefore $(I, \mathsf{Cod}(M/I)) \models \WKL_0$. 

The stages of the form $s = 5i+2$ ensure that $Z_i \cap I$ is infinite in $I$ for every $i$ (and in particular $Z \cap I$ is infinite in $I$). 

The stages of the form $s = 5i+1$ ensure that $(I, \mathsf{Cod}(M/I)) \models \RT^1$ (and therefore $\BSig_2$): let $g : I \to k \in \mathsf{Cod}(M/I)$, there exists an index $i \in \omega$ such that $g = g_i \cap I$ (and therefore $g = g_i \uh I$) and since $k \in I$, we can take the $i$ to be large enough for $\min Z_{5i+1}$ to be bigger than $k$ (Since every such function appears infinitely many times in the enumeration). By construction $Z_{5i+2}$ is $g_i$-homogeneous, so $Z_{5i+2} \cap I$ is an element of $\mathsf{Cod}(M/I)$ that is $g$-homogeneous and infinite in $I$.

The stages of the form $s = 5i +3$ ensure that $(I, \mathsf{Cod}(M/I),\param^{\M},\Param^{\M} \cap I) \models T$ as for every $k \in I$, there exists an index $i$ such that $k < \min Z_{5i + 3}$ and therefore $\forall x < k \exists y \forall z < \max Z_{5i + 4} \ \theta(x,y,z)$, so $(I, \mathsf{Cod}(M/I),\param^{\M},\Param^{\M} \cap I) \models \forall x < k \exists y \forall z \ \theta(x,y,z)$ (since $\max Z_{5i+4} > I$), so $(I, \mathsf{Cod}(M/I)) \models T$.

The stages of the form $s = 5i + 4$ ensure that $(I, \mathsf{Cod}(M/I)) \models \Gamma$: Let $f : [I]^n \to k \in \mathsf{Cod}(M/I)$, there exists an index $i \in \omega$ such that $f = f_i \cap I$ (and therefore $f = f_i \uh I$). By construction, $M \models \Psi(f_i, Z_{5i+5})$ and thus $(I, \mathsf{Cod}(M/I)) \models \Psi(f, Z_{5i+5} \cap I)$ and since $Z_{5i+5} \cap I$ is infinite in $I$ we have $(I, \mathsf{Cod}(M/I)) \models \Gamma$.
\end{proof}

We can now define a Paris-Harrington-like principle to which our theorem will be $\forall \Pi^0_4$ conservative over~$\RCA_0 + \BSig_2$.

\begin{definition}[Paris-Harrington principle for density$(T)$]
Fix a $\Sigma^0_0$-formula $\theta(X \uh_z, t, x, y, z)$ with exactly the displayed free variables and some $n \in \omega$. Let~\emph{$n$-PH${}_\theta$($\Gamma$)} be the statement 

\qt{For every~$\param, b$ and every set $\Param$, if $\forall x \exists y \forall z \theta(\Param\uh_z, \param, x, y, z)$ then there is an $n$-dense$(\forall x \exists y \forall z \theta(\Param \uh_z, \param, x, y, z), \Gamma)$ set $X$ such that~$\min X > b$.}
\end{definition}

The following proposition shows that this Paris-Harrington-like statement can be actually proven by~$\Gamma$, with the help of compactness.

\begin{proposition}\label[proposition]{prop:ph-density}
Fix a $\Sigma^0_0$-formula $\theta(X \uh_z, t, x, y, z)$ with exactly the displayed free variables.
Then for every~$n \in \omega$, $\WKL_0 + \BSig_2 + \Gamma \vdash n\mbox{-}\mathsf{PH}_\theta(\Gamma)$.
\end{proposition}
\begin{proof}
Write $T(X,t) \equiv \forall x \exists y \forall z \theta(X \uh_z,t,x,y,z)$ and $\Gamma \equiv (\forall f : [\mathbb{N}]^n \to k)(\exists Y)(Y $ is infinite $ \wedge \Psi(f,Y))$.

Fix $\param$ and $\Param$ a set and assume $T(\Param,\param)$. By external induction on $n \in \omega$, we will prove the stronger property that \qt{$n$-density$(T(\Param,\param), \Gamma)$ is a largeness notion}.

Case $n = 0$. By \Cref{prop:largeness-bsig2-largeness}, $\RCA_0 + \BSig_2 + T(\Param,\param)$ proves that $\bbomega\mbox{-largeness}(T(\Param,\param))$ is a largeness notion. 

Case $n > 0$. Suppose the property to be true at rank $n-1$ and fix $Y$ an infinite set. By the standard compactness argument (see \Cref{lem:compactness-big2-wkl}), available within $\WKL_0 + \Gamma$, and by \Cref{lem:rtnk-like-density}, there is a depth $d_0$ such that for every $f : [[0,\dots, d_0) \cap Y]^n \to k$ there exists a $Z \subseteq [0, d_0) \cap Y$ such that $\Psi(f,Z)$ and $Z$ is $(n-1)$-dense$(T(\Param,\param), \Gamma)$.

Let~$\ell = \min Y$. Consider the tree of all $\ell$-partition of $Y$ such that no elements of the partition contains a $(n-1)$-dense$(T(\Param,\param), \Gamma)$ set. By the inductive hypothesis and $\BSig_2$, this tree has no infinite branch, so, by $\WKL_0$, there is a depth $d_1$ such that for every partition $Z_0 \sqcup \dots \sqcup Z_{\ell-1} = Y \cap [0, \dots, d_1)$, one of the $Z_i$ is $(n-1)$-dense$(T(\Param,\param), \Gamma)$.

By $\BSig_2$ and our assumption that $\forall x \exists y \forall z \theta(\Param\uh_z,\param,x,y,z)$, there exists a bound $b$ such that $\forall x < \min Y \exists y < b \forall z \theta(\Param\uh_z,\param,x,y,z)$. By the inductive hypothesis, there exists some $d_2$ such that $Y \cap [b, b+1, \dots, d_2)$ is $(n-1)$-dense$(T(\Param,\param), \Gamma)$. 

Finally, take $d = \max \{d_0,d_1,d_2\}$, the set $Y \cap [0, \dots, d-1)$ is $n$-dense$(T(\Param,\param), \Gamma)$ by definition of $d_0,d_1,d_2$.
We conclude by induction.
\end{proof}

We are now ready to prove the general theorem which relates $\Gamma$ to its corresponding Paris-Harrington-like theory.

\begin{theorem}\label[theorem]{thm:ph-conservative}
$\WKL_0 + \BSig_2 + \Gamma$ is $\forall \Pi^0_4$ conservative over~$\RCA_0 + \{ n\mbox{-}\mathsf{PH}_\theta(\Gamma) : n \in \omega, \theta \in \Sigma^0_0 \}$
\end{theorem}
\begin{proof}
Assume $\RCA_0 + \{ n\mbox{-}\mathsf{PH}_\theta(\Gamma) : n \in \omega, \theta \in \Sigma^0_0 \} \not \vdash \forall X \forall t \exists x \forall y \exists z \eta(X\uh_z,t,x,y,z)$ with $\eta$ a $\Sigma_0^0$ formula. By compactness, completeness and the Löwenheim-Skolem theorem, there exists some countable model $$\M = (M,S) \models \RCA_0 + c\mbox{-}\mathsf{PH}_{\neg \eta}(\Gamma) + \{c > n : n \in \omega \} + \forall x \exists y \forall z \neg \eta(\Param \uh_z,\param,x,y,z)$$ where $\Param \in S, \param \in M$ and $c$ is a new constant symbol.

$\M \models \forall x \exists y \forall z \neg \eta(\Param,\param,x,y,z) + c\mbox{-}\mathsf{PH}_{\neg \eta}(\Gamma)$ so there exists a $c$-dense$(T(\Param,\param), \Gamma)$ set $X$, where $T(\Param,\param) \equiv \forall x \exists y \forall z \neg \eta(\Param\uh_z,\param,x,y,z)$.

By \Cref{prop:initial-segment-wkl-bsig2-t-gamma-density}, there exists an initial segment $I$ of $M$ such that $(I, \mathsf{Cod}(M/I)) \models \WKL_0 + \BSig_2 + \Gamma + T(\Param \cap I, \param)$. Therefore $\WKL_0 + \BSig_2 + \Gamma \not \vdash \forall X \forall t \exists x \forall y \exists z \eta(X\uh_z,t,x,y,z)$. \\

The converse follows from \Cref{prop:ph-density}.

\end{proof}

\subsection{Largeness$(T)$ for $\RT$-like statements}\label[section]{sect:largeness-ramsey-like}

Thanks to \Cref{thm:ph-conservative}, our main theorem is reduced to proving $n\mbox{-}\mathsf{PH}_\theta(\RT^2_2)$ over~$\RCA_0 + \BSig_2$ for every~$n \in \omega$ and every $\Sigma^0_0$ formula~$\theta$. Unfolding the definitions, given a set~$\Param$ and $\param$, letting~$T(\Param,\param) \equiv \forall x \exists y \forall z \theta(\Param \uh_z, \param, x, y, z)$, the goal is to prove over~$\RCA_0 + \BSig_2 + T(\Param,\param)$ the existence of $n$-dense$(T(\Param,\param), \RT^2_2)$ sets. However, the notion of $n$-density is not easy to manipulate. Thanks to our generalized Parsons theorem for provability over~$\WKL_0 + \BSig_2 + T$, one can handle one application of~$\RT^2_2$ at a time, but in return the solution has to be sufficiently large in the sense of $\bbomega^n$-largeness$(T)$.

\begin{definition}
Fix $r,s \in \omega$ and an $\RT$-like-$\Pi^1_2$-statement $\Gamma \equiv (\forall f : [\mathbb{N}]^n \to k)(\exists Y)(Y \mbox{is infinite} \wedge \Psi(f,Y))$. A set $Z \subseteq_{fin} \mathbb{N}$ is said to be $\bbomega^r\cdot s\mbox{-large}(T,\Gamma)$ if for any $f : [Z]^n \to k$, there is an $(\bbomega^r\cdot s)\mbox{-large}(T)$ set $Y \subseteq Z$ such that $\Psi(f, Y)$ holds.
\end{definition}


The following proposition relates density to largeness:

\begin{proposition}\label[proposition]{prop:largeness-to-density}
Let $\Gamma$ be an $\RT$-like statement and $T$ be a $\Pi^0_3$ formula.
Suppose $\bbomega^n$-largeness$(T, \Gamma)$ is a largeness notion provably in~$\RCA_0 + \BSig_2 + T$ for every~$n \in \omega$.
Then $n$-density$(T,\Gamma)$ is a largeness notion provably in~$\RCA_0 + \BSig_2 + T$ for every~$n \in \omega$.
\end{proposition}
\begin{proof}
Write $T \equiv \forall x\exists y \forall z \theta(x,y,z)$ and $\Gamma \equiv (\forall f : [\mathbb{N}]^n \to k)(\exists Y)(Y $ is infinite $ \wedge \Psi(f,Y))$.

By \Cref{thm:generalized-parsons-p4-bsig2-largeness}, for every $n \in \omega$, since $\bbomega^n$-largeness$(T, \Gamma)$ is a largeness notion provably in~$\RCA_0 + \BSig_2 + T$, there exists some $\ell_n \in \omega$, such that $$\ISig_1 \vdash \forall X (X \mbox{ is } \bbomega^{\ell_n}\mbox{-large}(T) \mbox{ and exp-sparse} \rightarrow X \mbox{ is } \bbomega^n\mbox{-large}(T, \Gamma))$$

Consider the following inductively defined sequence: $k_0 = 1$ and $k_{n+1} = \max \{2k_n, \ell_{k_n}, k_n+1\}$. We claim that $$\ISig_1 \vdash \forall X (X \mbox{ is } \bbomega^{k_n}\mbox{-large}(T) \mbox{ and exp-sparse} \rightarrow X \mbox{ is } n\mbox{-dense}(T,\Gamma))$$

By \Cref{prop:largeness-bsig2-largeness-sparsity}, $\RCA_0 + \BSig_2 + T$ proves that exp-sparse $\bbomega^k$-largeness$(T)$ is a largeness notion for every $k \in \omega$, so if the claim is valid then $n$-density$(T,\Gamma)$ is a largeness notion provably in~$\RCA_0 + \BSig_2 + T$ for every~$n \in \omega$. \\

We prove the claim by external induction on $n \in \omega$:

Case $n = 0$, every $\bbomega^{1}\mbox{-large}(T)$ set is $0\mbox{-dense}(T,\Gamma)$

Case $n > 0$, assume the property to be true at rank $n - 1$. Let $X$ be $\bbomega^{k_n}\mbox{-large}(T)$ and exp-sparse. We need to check that $X$ is $n\mbox{-dense}(T,\Gamma)$. Since by \Cref{rem:density}, item (b) follows from (c), then we prove (a), (c) and (d) of \Cref{def:density} :
\begin{itemize}
    \item[(a)] Since $X$ is $\bbomega^{\ell_{k_{n-1}}}\mbox{-large}(T)$ and exp-sparse, $X$ is $\bbomega^{k_{n-1}}\mbox{-large}(T, \Gamma)$, so for any $f : [X]^n \to k$, there is an $(n-1)\mbox{-dense}(T, \Gamma)$ subset $Y \subseteq X$ such that $\Psi(f,Y)$ holds.
    \item[(c)] Since $X$ is $\bbomega^{2k_{n-1}}\mbox{-large}(T)$ and exp-sparse, by \Cref{lem:largeness-rt1k}, for every coloring $f : X \to \min X$, there is an $\bbomega^{k_{n-1}}\mbox{-large}(T)$ subset $Y \subseteq X$ which is $f$-homogeneous. By induction hypothesis, $Y$ is $(n-1)\mbox{-dense}(T,\Gamma)$.
    \item[(d)] Since $X$ is $\bbomega^{k_{n-1} + 1}\mbox{-large}(T)$, there exists $X_0 < \dots < X_{\min X - 1}$ $(n-1)\mbox{-dense}(T, \Gamma)$ subsets of $X$ that are pairwise $T$-apart. So $\forall x < \min X \exists y < \min X_1 \forall z < \max X_1 \theta(x,y,z)$ (since $\min X \leq \max X_0$ and $X_0$ and $X_1$ are $T$-apart).
\end{itemize}

Therefore, $X$ is indeed $n\mbox{-dense}(T,\Gamma)$. We conclude by induction.

\end{proof}

The following theorem is the one which will actually be used in our applications in~\Cref{sect:applications}.

\begin{theorem}\label[theorem]{thm:largeness-to-conservation}
Let $\Gamma$ be an $\RT$-like. If $\bbomega^n$-largeness$(T, \Gamma)$ is a largeness notion provably in~$\RCA_0 + \BSig_2 + T$ for every~$n \in \omega$ and every $\Pi^0_3$ formula~$T$, then $\WKL_0 + \BSig_2 + \Gamma$ is a $\forall\Pi_4^0$-conservative extension of $\RCA_0 + \BSig_2$.
\end{theorem}
\begin{proof}
Fix some~$n \in \omega$ and $\theta(\Param\uh_z, \param, x, y, z)$ a $\Sigma^0_0$ formula.
Let us show that $\RCA_0 + \BSig_2 \vdash n\mbox{-}\mathsf{PH}_\theta(\Gamma)$.
Fix a model~$\M = (M, S) \models \RCA_0 + \BSig_2$, some~$\param, b \in M$ and some set~$\Param \in S$ such that $\M \models \forall x \exists y \forall z \theta(\Param\uh_z, \param, x, y, z)$ holds. Let~$T(\Param,\param) \equiv \forall x \exists y \forall z \theta(\Param\uh_z, \param, x, y, z)$. By \Cref{prop:largeness-to-density}, $n$-density$(T(\Param,\param),\Gamma)$ is a largeness notion provably in~$\RCA_0 + \BSig_2 + T(\Param,\param)$ for every~$n \in \omega$, thus $n$-density$(T(\Param,\param),\Gamma)$ is a largeness notion in~$\M$. Since~$(b, \infty)$ is an $M$-infinite set, then there is an $n$-dense$(T(\Param,\param), \Gamma)$ set~$X \in S$ such that $\min X > b$. So $\M \models n\mbox{-}\mathsf{PH}_\theta(\Gamma)$. Thus, by \Cref{thm:ph-conservative}, $\WKL_0 + \Gamma$ is a $\forall\Pi_4^0$-conservative extension of $\RCA_0 + \BSig_2$.
\end{proof}







\section{$\Pi^1_1$ conservation of the grouping principle}\label[section]{sect:grouping-principle}

The grouping principle was defined by Patey and Yokoyama~\cite{patey2018proof}
as a combinatorial principle needed for an inductive construction of $\bbomega^n$-large solutions to the Erdos-Moser theorem. They proved that the grouping principle is $\Pi^0_3$ conservative over~$\RCA_0$ and that it implies $\BSig_2$. In this section, we prove that this principle is actually $\Pi^1_1$ conservative over~$\RCA_0 + \BSig_2$.

There exist two notions of grouping: a finitary one ($\FGP^n_k$) and an infinitary one ($\GP^n_k$), the former following from the latter. We first define the infinitary version.

\begin{definition}[$\L$-grouping]
Let $\L$ be a largeness notion and $f : [\mathbb{N}]^n \to k$ a coloring.
A (finite or infinite) sequence of $\L$-large sets $F_0 < F_1 < \dots$ of length $k \in \NN \cup \{\NN\}$ is an $\L$-grouping for~$f$ if for every~$H \in [k]^n$, $f$ is monochromatic on $\prod_{i \in H} F_i$.
\end{definition}

In particular, for~$n = 1$, $F_0 < F_1 < \dots$ is an $\L$-grouping if for every~$i$, $F_i$ is $f$-homogeneous. 

For $\vec{f} = f_0, \dots, f_{k-1}$ a finite sequence of colorings of the same arity, we say that $F_0 < F_1 < \dots$ is an $\L$-grouping for $\vec{f}$ if $F_0 < F_1 < \dots$ is an $\L$-grouping for every $f_i$ (or, equivalently, if it is an $\L$-grouping for the product coloring $f : x \mapsto (f_0(x), \dots, f_{k-1}(x)$).

\begin{statement}[Grouping principle]
$\GP^n_k$: For every largeness notion~$\L$ and every coloring $f : [\NN]^n \to k$, there is an infinite $\L$-grouping.
\end{statement}

The remainder of this section is dedicated to the proof of the following theorem.

\begin{theorem}\label[theorem]{th:gp22-wkl-pi11-conservation}
$\WKL_0 + \GP_2^2$ is $\Pi_1^1$-conservative over $\RCA_0 + \BSig_2$.     
\end{theorem}

Following the reverse mathematical practice, $\GP^2_2$ can be decomposed into the cohesiveness principle ($\COH$) and a $\Delta^0_2$ version of $\GP^1_2$ ($\Delta^0_2\mbox{-}\GP^1_2$).

\begin{statement}
$\Delta^0_2\mbox{-}\GP^1_2$: For every largeness notion~$\L$ and every $\Delta^0_2$ set~$A$, there is an infinite $\L$-grouping for~$A$.
\end{statement}

An infinite set~$C \subseteq \NN$ is \emph{cohesive} for an infinite sequence of sets~$R_0, R_1, \dots$ if for every~$x \in \NN$, $C \subseteq^* R_x$ or $C \subseteq^* \overline{R}_x$, where $\subseteq^*$ denotes inclusion up to finitely many elements.

\begin{statement}[Cohesiveness]
$\COH$: Every infinite sequence of sets admits an infinite cohesive set.
\end{statement}

The following decomposition can be considered as folkore:

\begin{lemma}[Folklore]\label[lemma]{lem:gp-decomposition}
$\RCA_0 + \BSig_2 \vdash \COH + \Delta^0_2\mbox{-}\GP^1_2 \rightarrow \GP^2_2$.    
\end{lemma}
\begin{proof}
Fix some coloring $f : [\NN]^2 \to 2$ and largeness notion~$\L$. By $\COH$,
letting $R_x = \{ y : f(x, y) = 1 \}$, there is an infinite set~$C$ which is cohesive for~$R_0, R_1, \dots$ In particular, $\lim_{y \in C} f(x, y)$ exists for every~$x \in M$. Say~$C = \{ c_0 < c_1 < \dots \}$. Let~$A(x) = \lim_y f(c_x, c_y)$. By $\Delta^0_2\mbox{-}\GP^1_2$, there is an infinite $\L$-grouping~$F_0 < F_1 <  \dots$ for~$A$. For every~$s \in M$, let~$G_s = \{ c_x : x \in F_s \}$. Then for every~$s \in M$, either $\forall x \in F_s \lim_{y \in C} f(x, y) = 0$, or  $\forall x \in F_s \lim_{y \in C} f(x, y) = 1$. Then, using $\BSig_2$, by computably thinning out the sequence, we obtain an $M$-infinite grouping for~$f$.
\end{proof}

Given a model $\M = (M, S) \models \RCA_0$ and a set~$G \subseteq M$, we write $\M[G]$ for the model whose first-order part is~$M$ and whose second-order part consists of the $\Delta^0_1$-definable sets with parameters in $\M \cup \{G\}$. If $\M \cup \{G\} \models \ISig_1$, then $\M[G] \models \ISig_1$.
By the amalgamation theorem of Yokoyama~\cite[Theorem 2.2]{yokoyama2009conservativity}
and the fact that $\WKL_0$ and $\COH$ are both $\Pi^1_1$ conservative over~$\RCA_0 + \BSig_2$, \Cref{th:gp22-wkl-pi11-conservation} is a direct consequence of the following proposition. 

\begin{proposition}\label[proposition]{prop:gp22-low}
Consider a countable model $\M = (M,S) \models \RCA_0 + \BSig_2$ topped by a set $Y \in S$. For every largeness notion $\L \in S$ and every $\Delta_2^0(Y)$-set $A \subseteq M$, there exists a $G \subseteq M$ such that:
\begin{itemize}
    \item $G$ is an $M$-infinite $\L$-grouping for $A$ or $G$ is a witness that $\L$ is not a largeness notion (i.e. an $M$-unbounded set with no $M$-finite subset in $\L$).
    \item $Y' \geq (Y \oplus G)'$
    \item $\M[G] \models \RCA_0 + \BSig_2$. 
\end{itemize}
\end{proposition}

\begin{proof}

Fix a uniform enumeration of all $Y$-primitive recursive non-empty tree functionals $T_0, T_1, \dots$, 
that is, for every~$n \in M$ and every~$X$, $T_n^X$ is an $\M$-infinite tree. Let
$$
\C = \{ \bigoplus_a X_a : \forall a \forall n\ X_{\langle a, n\rangle} \in [T_n^{Y \oplus X_0 \oplus \dots \oplus X_a}] \}
$$
The class $\C$ is $\Pi^0_1(Y)$ and non-empty. There exists a primitive $Y$-recursive tree whose infinite path are the elements of $\C$, so by \cite[Corollary I.3.10(3)]{hajek1998metamathematics}, there is a set $Q = \bigoplus_a X_a \in \C$ such that $Y' \geq_T (Q \oplus Y)'$ and $\M[Q] \models \RCA_0 + \BSig_2$. Let~$\mathcal{N} = (M, \{ X_a : a \in M \})$. Note that $\mathcal{N}$ is an $\omega$-extension of~$\M$ and that $\mathcal{N} \models \WKL_0 + \BSig_2$.
Assume $\L$ is a largeness notion within~$\mathcal{N}$, otherwise we are done.

\begin{definition}[Condition]
A \emph{condition} is a tuple $(\vec{F}, \vec{g})$ such that
$\vec{F}$ is a finite $\L$-grouping for~$A$, and $\langle g_0, \dots, g_{p-1}\rangle \in \mathcal{N}$ are  2-colorings of~$M$.
\end{definition}

Note that every set~$X \in \mathcal{N}$ has a code with respect to~$Q$, that is, some~$a \in M$ such that $X = X_a$ (where $\bigoplus X_a = Q$). Thus, since $\langle g_0, \dots, g_{p-1}\rangle \in \mathcal{N}$, a condition $(\vec{F}, \vec{g})$ can be represented by an integer in~$M$. For simplicity of notation, we identify a condition with its code.

\begin{definition}[Extension]
A condition $d = (\vec{E}, \vec{h})$ \emph{extends} a condition $c = (\vec{F}, \vec{g})$ (written $d\leq c$) if:
\begin{itemize}
    \item $\vec{g}$ is a subsequence of $\vec{h}$
    \item $\vec{E}$ is of the form $\vec{F}, H_0, \dots, H_{s-1}$
    \item $H_0, \dots, H_{s-1}$ is an $\L$-grouping for $\vec{g}$.
\end{itemize}
\end{definition}

\begin{definition}[Forcing relation]
Let $\Phi_e$ be a Turing functional and $c = (\vec{F}, \vec{g})$ a condition:
\begin{itemize}
    \item $c \Vdash \Phi_e^{Y \oplus G} \downarrow$ if $\Phi_e^{Y \oplus \vec{F}} \downarrow$
    \item $c \Vdash \Phi_e^{Y \oplus G} \uparrow$ if $\Phi_e^{Y \oplus (\vec{F},\vec{E})} \uparrow$ for every finite $\L$-grouping~$\vec{E}$ for $\vec{g}$ such that $\max \vec{F} < \min \vec{E}$.
\end{itemize}
The relation $c \Vdash \Phi_e^{Y \oplus G} \downarrow$ is $\Delta_0^0(Y)$ and $c \Vdash \Phi_e^{Y \oplus G} \uparrow$ is $\Pi^0_1(Y \oplus \vec{g})$.
\end{definition}

\begin{lemma}\label[lemma]{lem:SGP22-forcing-jump}
Let $c = (\vec{F}, \vec{g})$ be a condition. For every~$k \in M$, there exists some $\M$-finite $\sigma' \in 2^k$ and an extension $d = (\vec{E}, \vec{h}) \leq c$ such that $d \Vdash (G \oplus Y)' \uh_k = \sigma'$ (Where $d \Vdash (G \oplus Y)' \uh_k = \sigma'$ is the $\Pi_1^0(Y \oplus \vec{h})$  relation meaning that if $\sigma'(e) = 0$ then $d \Vdash \Phi_e^{G \oplus Y} \uparrow$ and if $\sigma'(e) = 1$ then $d \Vdash \Phi_e^{G \oplus Y} \downarrow$).
\end{lemma}

\begin{proof}
For $s < k$, let $\varphi(s)$ be the formula that holds if for every $(k-s)$-tuple of colorings $\vec{f} = f_0, \dots, f_{k-s-1}$ there exists a set of indexes $e_0 < \dots < e_{s-1} < k$ and a finite $\L$-grouping $\vec{H}$ for $(\vec{g},\vec{f})$ such that $\max \vec{F} < \min \vec{H}$ and $\Phi_{e_i}^{Y \oplus (\vec{F}, \vec{H})} \downarrow$ for all $i < s$. By \Cref{lem:compactness-big2-wkl}, $\varphi(s)$ is $\Sigma_1^0(Y \oplus \vec{g})$.

Let $V = \{s < k : \varphi(s)\}$. $V$ is $\M$-finite since $\mathcal{N} \models \ISig_1$ and $V$ contains $0$. Therefore, one of the following cases holds: \\

Case 1: $k-1 \in V$. In this case, by considering the coloring induced by $A$, by \Cref{lem:compactness-big2-wkl}, there exists a finite $\L$-grouping $\vec{H}$ for $(\vec{g},A)$ with $\max \vec{F} < \min \vec{H}$ and such that for every index $e < k$, $\Phi_e^{Y \oplus (\vec{F}, \vec{H})} \downarrow$. In this case, $((\vec{F},\vec{H}), \vec{g}) \Vdash (G \oplus Y)' \uh k = 11\dots11$. \\

Case 2: there is some $s < k -1$ with $s \in V$ but $s+1 \notin V$. Consider the following $\Pi^0_1(Y \oplus \vec{g})$ class: $\C$ is the class of all $(k-s-1)$-tuple of colorings $\vec{f} = f_0, \dots, f_{k-s-2}$ such that for every set of indexes $e_0 < \dots < e_{s-2} < k$ and every finite $\L$-grouping $\vec{H}$ for $(\vec{g},\vec{f})$ satisfying $\max \vec{F} < \min \vec{H}$, there exists a $i < s - 1$ such that $\Phi_{e_i}^{Y \oplus (\vec{F}, \vec{H})} \uparrow$. 
Since $s+1 \notin V$, the class $\C$ is non-empty. Moreover, since~$Y, \vec{g} \in \mathcal{N} \models \WKL_0$, there is some~$\vec{f} \in \C \cap \mathcal{N}$.

Then $\vec{f},A$ is a $(k-s)$-uple of colorings and since $s \in V$, there exists a set of indexes $e_0 < \dots < e_{s-1} < k$ and a finite $\L$-grouping $\vec{H}$ for $(\vec{g},\vec{f},A)$ such that $\max \vec{F} < \min \vec{H}$ and $\Phi_{e_i}^{Y \oplus (\vec{F}, \vec{H})} \downarrow$ for all $i < s$.

Letting $d = ((\vec{F}, \vec{H}), (\vec{g},\vec{f}))$ and $\sigma' \in 2^k$ the binary encoding of the set $\{e_0, \dots, e_{s-1}\}$, we have that $d$ extends $c$ and $d \Vdash (G \oplus Y)' \uh k = \sigma'$: By choice of $\vec{H}$ and the $e_i$, it is clear that $d \Vdash \Phi_e^{G \oplus Y} \downarrow$ for every $e < k$ such that $\sigma'(e) = 1$ and by choice of $\vec{f}$, for every $e < k$ such that $\sigma'(e) = 0$ since $\{e_0, \dots, e_{s-1}, e\}$ is of cardinality $s + 1$, no extension of $d$ can force $\Phi_e^{G \oplus Y}$ to halt (and therefore $d \Vdash \Phi_e^{G \oplus Y} \uparrow$).
    
\end{proof}

\begin{lemma}\label[lemma]{lem:SGP22-forcing-infinite}
    Let $c = (\vec{F}, \vec{g})$ be a condition. For every $k \in M$, there exists an extension $d = (\vec{E}, \vec{g}) \leq c$ such that $\max \vec{E} > k$.
\end{lemma}

\begin{proof}
Consider the $\Pi_1^0(Y \oplus \vec{g})$-class $\mathcal{D}$ of all colorings $h$ such that there is no $\L$-grouping $\vec{E}$ with $\max \vec{E} > k$ for $\vec{g},h$ with $\vec{E}$ is of the form $\vec{F}, H_0, \dots, H_{s-1}$.

If this class was non-empty, then since~$Y, \vec{g} \in \mathcal{N} \models \WKL_0$, then there is an element $f \in \mathcal{D} \cap \mathcal{N}$. By $\RT^1$, there exists an $\mathcal{N}$-infinite set $X$ homogenous for $\vec{g},f$ such that $X \cap [k,t]_{\mathbb{N}}$ is not $\L$-large for any $t > k$ contradicting our assumption on the largeness of $\L$ within~$\mathcal{N}$. 

So the class $\mathcal{D}$ is empty and in particular does not contain the coloring $A$. Therefore, such an extension $d$ exists.
\end{proof}

\textbf{Construction}.
We will build a decreasing sequence $(\vec{F}_s, \vec{g}_s)$ of conditions and then take for $G$ the union of the $\vec{F}_s$. We will also build an increasing sequence $(\sigma'_s)$ such that $(G \oplus Y)'$ will be the union of the $\sigma'_s$. Initially, we take $\vec{F}_0 = \vec{g}'_0 = \emptyset$ and $\sigma'_0 = \epsilon$, and during the construction, we will ensure that we have $|\vec{F}_s|, |\vec{g}_s|, |\sigma'_s| \leq s$ at every stage. Each stage will be either of type $\R$ or of type $\S$. The stage $0$ is of type $\R$.

Assume that $(\vec{F}_s, \vec{g}_s)$ and $\sigma'_{s}$ are already defined. Let~$s_0 < s$ be the latest stage at which we switched the stage type. We have three cases.

Case 1: $s$ is of type~$\R$. If there exists some~$\vec{F}$ and $\vec{g}$ such that $|\vec{g}|, |\vec{F}| \leq s$ and some~$\sigma' \in 2^{s_0}$ such that $(\vec{F}, \vec{g}) \leq (\vec{F}_s, \vec{g}_s)$, and 
$(\vec{F}, \vec{g}) \Vdash (G \oplus Y)' \uh_{s_0} = \sigma'$, then let~$\vec{F}_{s+1} = \vec{F}$, $\vec{g}_{s+1} = \vec{g}$, $\sigma'_{s+1} = \sigma'$ and let~$s+1$ be of type~$\S$. Otherwise, the elements are left unchanged and we go to the next stage.

Case 2: $s$ is of type~$\S$. If there exists some~$\vec{F}$ and $\vec{g}$ such that $|\vec{g}|, |\vec{F}| \leq s$, $\max \vec{F} > s_0$ and $(\vec{F}, \vec{g}) \leq (\vec{F}_s, \vec{g}_s)$, then let $\vec{F}_{s+1} = \vec{F}$ and $\vec{g}_{s+1} = \vec{g}$ and let~$s+1$ be of type~$\R$. Otherwise, the elements are left unchanged and we go to the next stage.

This completes the construction.
\bigskip

\textbf{Verification}.
Since the size of $\vec{F}_s, \vec{g}_s$ and $\sigma'_s$ are bounded by $s$, there is a $\Delta_1^0(Y' \oplus Q)$-formula $\phi(s)$ stating that the construction can be pursued up to stage $s$. Our construction implies that the set $\{s|\phi(s)\}$ is a cut, so by $\IDel_1(Y' \oplus Q)$ (which follows from $\BSig_2$ in $\M[Q]$), the construction can be pursued at every stage. 

Let~$G = \bigcup_{s \in M} \sigma_s$.
By \Cref{lem:SGP22-forcing-jump} and \Cref{lem:SGP22-forcing-infinite},
each type of stage changes $\M$-infinitely often. Thus, $\{|\sigma'_s| : s \in M \}$ is $\M$-unbounded, so $(Y \oplus Q)' \geq_T (G \oplus Y)'$. Since~$Y' \geq_T (Y \oplus Q')$, then $Y' \geq_T(G \oplus Y)'$. In particular, $\M[G] \models \RCA_0 + \BSig_2$ and $G$ is $M$-infinite and is therefore an $M$-infinite $L$-grouping for $A$.
\end{proof}

We are now ready to prove the main theorem of this section, namely, that the grouping principle is $\Pi^1_1$ conservative over~$\RCA_0 + \BSig_2$. It can be alternatively proved by showing that $\Delta^0_2\mbox{-}\GP^1_2$ and $\COH$ both are $\Pi^1_1$-conservative over~$\RCA_0 + \BSig_2$, and using the amalgamation theorem of Yokoyama~\cite[Theorem 2.2]{yokoyama2009conservativity}. However, we give a direct argument for the sake of simplicity:

\begin{proof}[Proof of \Cref{th:gp22-wkl-pi11-conservation}]
Assume $\RCA_0 + \BSig_2 \not \vdash \forall X \phi(X)$ for $\phi$ arithmetical. Then, by completeness and the Löwenheim-Skolem theorem, there exists a model $\M = (M,S) \models \RCA_0 + \BSig_2 + \neg\phi(A)$ for an $A \in S$. We can furthermore assume that $\M$ is topped by $A$.

By \Cref{prop:gp22-low} and Chong, Slaman and Yang~\cite[Theorem 3.2]{chong2021pi11}, we can build the following sequence of countable topped model of $\RCA_0 + \BSig_2$: $\M_0 = \M \subseteq \M_1 \subseteq \M_2 \subseteq \dots$ satisfying that 
\begin{itemize}
    \item[(1)] for every $i$, every largeness notion $\L \in \M_i$ and every set $A$ that is $\Delta_2^0$ in $\M_i$, there is either a $j$ such that $\L$ is no longer a largeness notion in $\M_j$, or a set $G \in \M_j$ such that $G$ is an $M$-infinite $\L$-grouping for $A$.
    \item[(2)] for every $i$, every uniform sequence of sets $\vec{R} \in \M_i$, there is some~$j$ and an $M$-infinite $\vec{R}$-cohesive set~$C \in \M_j$.
\end{itemize}

Let $\M' = \bigcup_{n \in \omega} \M_n$. Clearly, $\M' \models \RCA_0 + \BSig_2 + \COH + \Delta^0_2\mbox{-}\GP^1_2 + \neg \phi(A)$. By \Cref{lem:gp-decomposition}, $\M' \models \GP^2_2$.
Therefore, $\GP_2^2$ is $\Pi_1^1$-conservative over $\RCA_0 + \BSig_2$.
\end{proof}

\begin{remark}
The above construction shows that $\RCA_0 + \BSig_2$ proves the $\ll_2$-basis theorem for $\Delta^0_2\mbox{-}\GP^1_2$ (in the sense of \cite{fiori2021isomorphism}), and thus $\GP^2_2$ also admits $\ll_2$-basis theorem within $\RCA_0 + \BSig_2$.
This implies that the above $\Pi_1^1$-conservation theorem admits poly-time proof transformation.
In general, an upcoming result by Ikari and Yokoyama (see \cite{Ikari-PhD, Ik-Yo}) shows that any $\RT$-like principle provable from $\WKL_0+\RT^2_2$ admits poly-time proof transformation if it is $\Pi_1^1$-conservative over $\RCA_0 + \BSig_2$.
\end{remark}

\subsection{Finite grouping principle with $T$-apartness}\label[section]{sect:finite-grouping-principle}

Thanks to the generalized Parsons theorem, we can turn $\Pi^1_1$ conservation of the grouping principle into a quantitative finitary version based on $\bbomega^n$-largeness$(T)$.

\begin{definition}[Finite grouping principle for $T$]
Let $\L_0, \L_1$ be largeness notions and $f : [X]^n \to k$ a coloring. A finite sequence of $\L_0$-large sets $F_0 < F_1 < \dots < F_{k-1}$ is an $(\L_0, \L_1)$-grouping$(T)$ for $f$ if: 
\begin{itemize}
    \item for any $H \subseteq_{fin} \mathbb{N}$, if $H \cap F_i \neq \emptyset$ for every $i < k$, then $H \in \L_1$, and,
    \item for every $H \in [k]^n$, $f$ is monochromatic on $\prod_{i \in H} F_i$
    \item for every $i < j < k-1$, $F_i$ and $F_{j}$ are $T$-apart.
\end{itemize}

Let $\pFGP{\L_0}{\L_1}^n_k(T)$ be the statement that for any infinite set $X \subseteq \mathbb{N}$, there exists some finite set $Y \subseteq X$ such that for any coloring $f: [Y]^n \to k$, there exists an $(\L_0, \L_1)$-grouping$(T)$ for $f$.
\end{definition}

Recall that $\param$ and $\Param$ are the first-order and second-order parameters appearing in the formula~$T$.

\begin{proposition}\label[proposition]{thm:T-FGP22}
Let $\L_0$ and $\L_1$ be $\Delta^{\param,\Param}_0$-definable largeness notions provably in $\RCA_0 + \BSig_2 + T$. Then $\RCA_0 + \BSig_2 + T$ proves $\pFGP{\L_0}{\L_1}^2_2(T)$.
\end{proposition}

\begin{proof}
$\forall \param \forall \Param(T \rightarrow \pFGP{\L_0}{\L_1}^2_2(T))$ is a $\Pi_1^1$ statement. So by $\Pi_1^1$ conservation of $\WKL_0 + \GP_2^2$ over $\RCA_0 + \BSig_2$, it is sufficient to prove the result using $\WKL_0 + \BSig_2 + T + \GP_2^2$.

By $\GP_2^2$, for any infinite set $X$ and for any coloring $f : [X]^2 \to 2$ there exists an infinite $\L_0$-grouping for $f$.
By $\BSig_2 + T$, we can furthermore assume the blocks of these groupings to be $T$-apart. 

For each of these grouping $Y_0 < Y_1 < \dots$, we can consider the following finitely branching tree $S$ of finite sequences $\sigma$ such that $\sigma(i) \in Y_i$ for all $i < |\sigma|$ and such that the finite set $\{\sigma(i) | i < |\sigma|\}$ is not $\L_1$-large. Since $\L_1$ is a largeness notion, the tree $S$ has no infinite branch, and therefore by $\WKL_0$, there is some bound $n$ on the depth of the tree. So $Y_0 < Y_1 < \dots < Y_{n-1}$ is an $(\L_0, \L_1)$-grouping$(T)$ for the corresponding coloring.

And again by $\WKL_0$, there exists some finite set $Y \subseteq X$ such that for any coloring $f : [Y]^2 \to 2$ there exists an $(\L_0, \L_1)$-grouping$(T)$ for $f$. 

So $\WKL_0 + \BSig_2 + T + \GP_2^2 \vdash \pFGP{\L_0}{\L_1}^2_2(T)$ and therefore $\RCA_0 + \BSig_2 + T \vdash \pFGP{\L_0}{\L_1}^2_2(T)$.
\end{proof}

We will use the finite grouping principle with $T$-apartness under the following form.
Note that in the following proposition, the integer~$n$ might depend not only on~$k$ and $\ell$, but also on~$T$.
However, by adapting the direct combinatorial proof of Ko\l odziejczyk and Yokoyama~\cite[Section 2.1]{kolo2020some} of the finitary grouping principle with explicit bounds, one obtains a bound which does not depend on~$T$ (see \Cref{rem:gp-explicit-bounds}).

\begin{proposition}\label[proposition]{prop:fgp-apartness}
For any $k,\ell \in \omega$, there exists $n \in \omega$ such that $\ISig_1$ proves that $\forall Z \finsub \NN$, if $Z \mbox{ is } \bbomega^n\mbox{-large}(T) \mbox{ and exp-sparse}$ then
$$\forall f : [Z]^2 \to 2, \mbox{ there exists an } (\bbomega^k\mbox{-large}(T), \bbomega^\ell\mbox{-large}(T))\mbox{-grouping}(T) \mbox{ for }f$$
\end{proposition}

\begin{proof}
Using \Cref{thm:generalized-parsons-p4-bsig2-largeness}, it is sufficient to show that $\WKL_0 + \BSig_2 + T$ proves that $\forall X$ if $X \mbox{ is infinite}$ then 
$$\exists Y \subseteq X \forall f : [Y]^2 \to 2, \mbox{ there exists an }(\bbomega^k\mbox{-large}(T), \bbomega^\ell\mbox{-large}(T))\mbox{-grouping}(T) \mbox{ for }f$$  

Which we have by \Cref{thm:T-FGP22} (And by \Cref{prop:largeness-bsig2-largeness}, $\bbomega^k\mbox{-largeness}(T)$ and $\bbomega^\ell\mbox{-largeness}(T)$ are largeness notions provably in $\RCA_0 + \BSig_2 + T$).
\end{proof}

\begin{remark}\label[remark]{rem:gp-explicit-bounds}
The proof of \Cref{prop:fgp-apartness} involves a generalized Parsons theorem (\Cref{thm:generalized-parsons-p4-bsig2-largeness}) which does not provide explicit bounds for the existence of a grouping. Ko\l odziejczyk and Yokoyama~\cite[Section 2.1]{kolo2020some} gave a direct combinatorial proof of the finitary grouping principle, with explicit bounds. 

Largeness$(T)$ shares many combinatorial features with the standard notion of largeness. However, there are some differences, which impact the explicit bounds of the pigeonhole lemma for largeness$(T)$ (\Cref{lem:largeness-rt1k}). Indeed, to obtain an $\bbomega^n$-large$(T)$ set after an application of~$\RT^1$, one starts with an $\bbomega^{2n}$-large$(T)$ set, while an $\bbomega^{n+1}$-large set is sufficient in the case of standard largeness. 

Propagating this difference to the proof of Ko\l odziejczyk and Yokoyama~\cite[Section 2.1]{kolo2020some}, one obtain (with exp-sparsity) $\bbomega^{2n}$-largeness$(T)$ for \cite[Lemma 2.5]{kolo2020some}, hence $\bbomega^{4n+1}$-largeness$(T)$ for \cite[Lemma 2.6]{kolo2020some}, $\bbomega^{16n+5}$-largeness$(T)$ for \cite[Lemma 2.7]{kolo2020some} and $\bbomega^{16^k\times(n+1)}$-largeness$(T)$ for \cite[Theorem 2.4]{kolo2020some}. 

For $f : X \to 2$ a coloring of singletons, the existence of an $(\bbomega^n,2)$-grouping$(T)$ for $g : (x,y) \mapsto f(x)$ yields an $\bbomega^n$-large$(T)$ $f$-homogeneous subset of $X$ (by taking the first block of the grouping). We shall see in \Cref{sect:conclusion} that for all $n \in \omega$ there exists some $\Pi_3^0$-formula $T$ and some $\bbomega^{2n-1}$-large$(T)$ set $X$ such that $X$ is not $\bbomega^n$-large$(T,\RT^1_2)$, therefore the bound obtained for the existence of an $(\bbomega^n\mbox{-large}(T),\bbomega^k\mbox{-large}(T))$-grouping$(T)$ is tight in the sense that it is not possible to obtain one with $X$ $\bbomega^{n + h(k)}$-large$(T)$ for some $h : \omega \to \omega$ as in \cite{kolo2020some}. 

\end{remark}

\section{Applications to $\RT$-like theorems}\label[section]{sect:applications}

In this section, we apply the framework developed in \Cref{sect:framework-pi04} to prove that the Erd\H{o}s-Moser theorem, the Ascending Descending Sequence principle and Ramsey's theorem for pairs are $\forall \Pi^0_4$ conservative over~$\RCA_0 + \BSig_2$. 

\subsection{$\forall \Pi^0_4$ conservation of~$\EM$}

As explained in \Cref{sect:prerequisites}, $\RT^2_2$ can be decomposed into two combinatorially simpler statements, namely, $\ADS$ and $\EM$. Thanks to the amalgamation theorem (\Cref{thm:amalgamation}), since~$\RCA_0 + \ADS$ is a $\Pi^1_1$ conservative extension of~$\RCA_0 + \BSig_2$, the heart of the main question lies in the first-order part of~$\EM$.

As in \Cref{prop:fgp-apartness}, the bound $k_n$ depends on~$n$, but also on~$T$ by default. However, the alternative combinatorial proof yields explicit bounds which do not depend on~$T$ (see \Cref{rem:em-explicit-bounds}).

\begin{proposition}\label[proposition]{prop:largeness-em}
For any $n \in \omega$, there exists some $k_n \in \omega$ such that $\ISig_1$ proves that $\forall X \subseteq_{fin} \mathbb{N}$ if $X$ is $\bbomega^{k_n}$-large$(T)$ and exp-sparse then $X$ is $\bbomega^{n}$-large$(T,\EM)$.
\end{proposition}

\begin{proof}
By external induction on $n$:

Case $n = 0$: Any $\bbomega^0$-large$(T)$ set is $\bbomega^0$-large$(T, \EM)$. Thus, take $k_0 = 0$.

Case $n > 0$, assume the property to be true at rank $n - 1$. By \Cref{prop:fgp-apartness}, there is some $k_n \in \omega$ such that if $X$ is $\bbomega^{k_n}$-large$(T)$ and exp-sparse, then 
$$\forall f : [X]^2 \to 2, \mbox{ there exists an } (\bbomega^{k_{n-1}}\mbox{-large}(T), \bbomega^{6}\mbox{-large})\mbox{-grouping}(T) \mbox{ for }f$$
Note that $\bbomega^6$-largeness is not $\bbomega^6$-largeness$(T)$ in the grouping above, but any $\bbomega^6$-large$(T)$ set is also $\bbomega^6$-large.

We need to check that $X$ is $\bbomega^{n}$-large$(T,\EM)$. Let $f : [X]^2 \to 2$ be an instance of $\EM$. There exists an $(\bbomega^{k_{n-1}}$-large$(T), \bbomega^{6}$-large)-grouping$(T)$ $X_0 < X_1 < \dots < X_{\ell-1}$ for $f$. 

By the inductive hypothesis, each block $X_i$ is $\bbomega^{n-1}$-large$(T,\EM)$, so for each $i < \ell$ there exists $Y_i \subseteq X_i$ which is $\bbomega^{n-1}$-large$(T)$ and transitive for $f$. By definition of a grouping, $f$ induces a tournament on the pairs of blocks and $\{\min Y_i : i < \ell\}$ is $\bbomega^{6}$-large and therefore $\bbomega$-large$(\EM)$ by Ketonen and Solovay (see \Cref{thm:ketonen-solovay}). 

Therefore, there is a subset $I \subseteq \{0, \dots, l-1\}$ such that $f$ is transitive on $\{\min Y_i : i \in I\}$ which is $\bbomega$-large. The set $Y = \bigcup_{i \in I} Y_i$ is $f$-transitive and $\bbomega^{n}$-large$(T)$.
We conclude by induction.
\end{proof}

\begin{remark}\label[remark]{rem:em-explicit-bounds}
One can propagate the explicit bounds of the grouping principle computed in \Cref{rem:gp-explicit-bounds}
to obtain an explicit bound for \Cref{prop:largeness-em}: if $X$ is $\bbomega^{(16^6 + 1)^n}$-large$(T)$, then it is $\bbomega^n$-large$(T,\EM)$.

Note that we obtain an exponential upper bound for largeness$(T, \EM)$ while the upper bound for regular largeness$(\EM)$ is polynomial. 
Since the bounds on the grouping principle are tight, one would need to get rid of the applications of the grouping principle to obtain a better upper bound, which seems unlikely.
\end{remark}

\begin{theorem}\label[theorem]{thm:em-pi04-conservation}
$\WKL_0 + \EM$ is $\forall \Pi_4^0$-conservative over $\RCA_0 + \BSig_2$.    
\end{theorem}

\begin{proof}
By \Cref{prop:largeness-em}, and \Cref{prop:largeness-bsig2-largeness-sparsity}, for all $n \in \omega$ and $T$ a $\Pi_3^0$-formula, $\RCA_0 + \BSig_2 + T$ proves that $\bbomega^{n}$-largeness$(T,\EM)$ is a largeness notion.
So by \Cref{thm:largeness-to-conservation}, $\EM$ is a $\forall\Pi_4^0$-conservative extension of $\RCA_0 + \BSig_2$.
\end{proof}

We are now ready to prove the main theorem of this article:

\begin{reptheorem}{thm:rt22-pi04-conservation}
$\WKL_0 + \RT^2_2$ is $\forall \Pi_4^0$-conservative over $\RCA_0 + \BSig_2$.     
\end{reptheorem}
\begin{proof}
Since $\ADS$ is $\Pi_1^1$-conservative over $\RCA_0 + \BSig_2$, it is in particular $\forall \Pi_4^0$-conservative. By the amalgamation theorem (\Cref{thm:amalgamation}) and the fact that $\RCA_0 \vdash \RT_2^2 \leftrightarrow \ADS + \EM$ (see Bovykin and Weiermann~\cite{bovykin2017strength}), remains to show that $\WKL_0 + \EM$ is also $\forall \Pi_4^0$-conservative over $\RCA_0 + \BSig_2$, which is \Cref{thm:em-pi04-conservation}.
\end{proof}

\subsection{$\forall \Pi^0_4$ conservation of~$\RT^2_2$}


In this section, we give an inductive proof in~$\RCA_0 + \BSig_2 + T$ that $\bbomega^n$-largeness$(T, \RT^2_2)$ is a largeness notion. We start with a technical lemma.


\begin{lemma}\label[lemma]{lem:grouping-to-homogeneous}
Fix~$n, k \in \omega$ with~$n \geq 2$, and suppose that $\ISig_1$ proves that every~$\bbomega^k$-large$(T)$ set is $\bbomega^{n-1}$-large$(T, \RT^2_2)$. Then there is some~$\ell \in \omega$ such that $\ISig_1$ proves if~$X$ is $\bbomega^\ell$-large$(T)$, then for every coloring $f : [X]^2 \to 2$ one of the following holds:
\begin{itemize}
    \item[(1)] There exists an $\bbomega^n$-large$(T)$ $f$-homogeneous set $Y \subseteq X$
    \item[(2)] There exists $\bbomega^{n-1}$-large$(T)$ sets $Y_0, Y_1 \subseteq X$ such that $Y_c$ is $f$-homogeneous for color~$c$ for each~$c < 2$.
\end{itemize}
\end{lemma}
\begin{proof}
By \Cref{prop:fgp-apartness}, there is some~$\ell \in \omega$ such that $\ISig_1$ proves if~$X$ is $\bbomega^\ell$-large$(T)$, then for every~$f : [X]^2 \to 2$, there is an $(\bbomega^{k}\mbox{-large}(T), \bbomega^{2 \cdot k}\mbox{-large}(T))$-grouping$(T)$ for~$f$.
Let~$X_0 < X_1 < \dots < X_{s-1}$ be this grouping for such an $f$.

By assumption, each~$X_i$ is $\bbomega^{n-1}$-large$(T, \RT^2_2)$, so for each~$i < s$, there exists some $\bbomega^{n-1}$-large$(T)$ set~$Z_i \subseteq X_i$ which is $f$-homogeneous for some color~$c_i < 2$. 

Since $\{\min Z_i : i < s \}$ is $(\bbomega^{2 \cdot k})\mbox{-large}(T)$ (by definition of a grouping), by \Cref{lem:largeness-rt1k}, there exists some color $c < 2$ such that $\{\min Z_i : i < s \wedge c_i = c\}$ is $\bbomega^k\mbox{-large}(T)$. Let $I = \{i < s : c_i = c\}$. 

The coloring $f$ induces a coloring on $\{\min Z_i : i \in I\}$ which is $\bbomega^k$-large$(T)$, so there is a subset $I' \subseteq I$ such that $\{\min Z_i : i \in I'\}$ is $\bbomega^{n-1}\mbox{-large}(T)$ and $f$-homogeneous for some color $d$. We have two cases:

\begin{itemize}
    \item If $c = d$, then, since~$n \geq 2$,  $\{\min Z_i : i \in I'\}$ is $\bbomega^1$-large$(T)$. In particular, $\card I' > \min Z_i$, and the sets $(Z_i)_{i \in I'}$ are pairwise $\theta$-apart and $\bbomega^{n-1}\mbox{-large}(T)$, so $Y = \bigcup_{i \in I'} Z_i$ is $f$-homogeneous and $\bbomega^{n}\mbox{-large}(T)$. We are in case $(1)$.
    \item If $c \neq d$, then, let $Y_c = Z_{\min I}$ and $Y_d = \{\min Z_i : i \in I'\}$. Both of these sets are $\bbomega^{n-1}\mbox{-large}(T)$ and $f$-homogeneous for a different color. We are in case $(2)$.
\end{itemize}
This completes the proof of \Cref{lem:grouping-to-homogeneous}.
\end{proof}

\begin{proposition}\label[proposition]{prop:grouping-to-homogeneous-2}
For any $n \in \omega$, there exists some $k_n \in \omega$ such that $\ISig_1$ proves that, for every $X \subseteq \mathbb{N}$, if $X$ is $\bbomega^{k_n}$-large$(T)$ and exp-sparse then $X$ is $\bbomega^{n}$-large$(T,\RT^2_2)$.
\end{proposition}
\begin{proof}
By external induction on $n$:

Case $n = 0$, take $k_0 = 0$, as a set is $\bbomega^0$-large$(T)$ iff it is non-empty, and any singleton element is homogeneous for any coloring. \\

Case $n = 1$, by \Cref{prop:fgp-apartness}, there is some~$k_1 \in \omega$ such that if $X$ is $\bbomega^{k_1}\mbox{-large}(T)$ and exp-sparse, then 
$$\forall f : [X]^2 \to 2, \mbox{ there exists an } (\bbomega^0\mbox{-large}, \bbomega^{6}\mbox{-large})\mbox{-grouping}(T) \mbox{ for }f$$

We need to check that such an $X$ is $\bbomega^{1}$-large$(T,\RT_2^2)$. Let $f : [X]^2 \to 2$ be a coloring. Let $X_0 < X_1 < \dots < X_{s-1}$ be an $(\bbomega^0\mbox{-large}, \bbomega^{6}\mbox{-large})\mbox{-grouping}(T)$ for $f$. Every $X_i$ is $\bbomega^0$-large and therefore contains an element $x_i$. By definition of a grouping$(T)$, the family $\{x_i: i < s\}$ is $\bbomega^6$-large and all the $x_i$'s are $T$-apart. By \Cref{thm:ketonen-solovay}, there is a subset $I \subseteq \{0, \dots, s-1\}$ such that $\{x_i: i \in  I\}$ is $\bbomega$-large (and therefore $\bbomega$-large$(T)$ since the $x_i$'s are $T$-apart) and $f$-homogeneous.  \\

Case $n > 1$, assume the property to be true at rank $n-1$. Define $\ell_n$ so that \Cref{lem:grouping-to-homogeneous} holds for~$n$ and $k_{n-1}$.
By \Cref{prop:fgp-apartness}, there is some~$k_n \in \omega$ such that if $X$ is $\bbomega^{k_n}\mbox{-large}(T)$ and exp-sparse, then 
$$\forall f : [X]^2 \to 2, \mbox{ there exists an } (\bbomega^{\ell_n}\mbox{-large}(T), \bbomega^{6}\mbox{-large})\mbox{-grouping}(T) \mbox{ for }f$$

We need to check that such an $X$ is $\bbomega^{n}$-large$(T,\RT_2^2)$. Let $f : [X]^2 \to 2$ be a coloring. Let $X_0 < X_1 < \dots < X_{s-1}$ be an $(\bbomega^{\ell_n}\mbox{-large}(T), \bbomega^{6}\mbox{-large})\mbox{-grouping}(T)$ for $f$. If for some $i < s$, there is an $\bbomega^{n}\mbox{-large}(T)$ $f$-homogeneous subset~$Y \subseteq X_i$, then we are done.
Thus, since each $X_i$ is $\bbomega^{\ell_n}\mbox{-large}(T)$, by \Cref{lem:grouping-to-homogeneous} there exists some $\bbomega^{n-1}\mbox{-large}(T)$ subsets $Y_{0,i}$ and $Y_{1,i}$ of $X_i$ that are $f$-homogeneous for the color $0$ and $1$, respectively. 

The coloring $f$ induces a coloring on the $\bbomega^6$-large set $\{ \max X_i : i < s\}$, so by \Cref{thm:ketonen-solovay}, there is a subset $I \subseteq \{0, \dots, s-1\}$ such that $\{ \max X_i : i \in I\}$ is $\bbomega$-large and $f$-homogeneous for some color~$c$. In particular, $\{\min Y_{c,i} : i \in I\}$ is $f$-homogeneous for color~$c$ and $\bbomega$-large. Since $\{\min Y_{c,i} : i \in I\}$ is $\bbomega$-large, $\card I > \min Y_{c,i}$. Furthermore, the sets $(Y_{c,i})_{i \in I}$ are $\bbomega^{n-1}\mbox{-large}(T)$, $f$-homogeneous for color~$c$ and pairwise $\theta$-apart, so $\bigcup_{i \in I} Y_{c,i}$ is $\bbomega^{n}\mbox{-large}(T)$ and $f$-homogeneous.
This completes the proof of \Cref{prop:grouping-to-homogeneous-2}.
\end{proof}

\begin{corollary}\label[corollary]{cor:tetration-bound-rt22}
    Let ${}^x y$ denotes the tetration of $x$ and $y$. Then, for every $n \in \bbomega$, $\ISig_1$ proves that, for every $X \finsub \NN$, if $X$ is $\bbomega^{{}^{(n+1)}16}$-large$(T)$, then it is $\bbomega^{n}$-large$(\RT_2^2,T)$. 
\end{corollary}

\begin{proof}
The proof of \Cref{prop:grouping-to-homogeneous-2} combined with the bounds for the grouping principle obtained in \cite[Remark 4.14]{houerou2026conservation} yield the following bounds: for $(k_n)_{n \in \omega}$ defined inductively by $k_0 = 0$, $k_1 = 16^6$ and $k_{n+1} = 16^6(16^{2k_n} (k_n + 1))$ for $n > 1$, $\ISig_1$ proves that every $\bbomega^{k_n}$-large$(T)$ set is $\bbomega^n$-large$(\RT_2^2,T)$.

We can then prove by induction on $n \in \omega$ that $3k_n + 7 \leq {}^{(n+1)}16$ for every $n \in \bbomega$, hence that $k_n \leq {}^{(n+1)}16$, which implies the statement of this corollary. The result holds for $n = 0$ and $n = 1$, and, assuming that it holds for some $n > 0$, we have:

$$3k_{n+1} + 7 = 3\cdot16^6(16^{2k_n} (k_n + 1)) + 7 < 16^{3k_n + 7} \leq {}^{(n+2)}16$$

This completes the proof of \Cref{cor:tetration-bound-rt22}.
\end{proof}

\section{A conservation theorem over~$\ISig_1$}\label[section]{sect:weak-formulas}

Since~$\BSig_2$ is a $\forall \Pi^0_4$ statement which is strictly stronger than~$\ISig_1$, then $\BSig_2$ (and \emph{a fortiori} $\RT^2_2$) is not a $\forall \Pi^0_4$ conservative extension of~$\ISig_1$. However, in this section, we shall give two proofs that it is the case for a particular kind of $\forall \Pi^0_4$ formulas.
The first proof uses the techniques developed in this article, while the second is based on the fact that every countable model of $\ISig_n$ admits a $\Sigma^0_{n+1}$-elementary cofinal extension which is a model of $\BSig_{n+1}$.

\begin{definition}
A formula is \emph{weakly $\forall \Pi^0_4$} if it is of the form
$$
\forall \Param\forall \param \exists x \forall y \exists x' < x \forall y' < y \exists z \theta(\Param\uh_z, \param, x', y', z)
$$
where~$\theta$ is a $\Sigma^0_0$ formula.
\end{definition}

Intuitively, weakly $\forall \Pi^0_4$ formulas are $\forall \Pi^0_4$ formulas for which the applications of $\BSig_2$ are \qt{hardcoded} in the syntax of the formula.

\begin{proposition}
Over~$\RCA_0 + \BSig_2$, every $\forall \Pi^0_4$ formula is equivalent to a weakly $\forall \Pi^0_4$ formula.
\end{proposition}
\begin{proof}
Let~$S = \forall \Param \forall \param \exists x \forall y \exists z \theta(\Param\uh_z, \param, x, y, z)$ be a $\forall \Pi^0_4$ formula, where $\theta$ is $\Sigma^0_0$.
Let~$S' = \forall \Param \forall \param \exists x \forall y \exists x' < x \forall y' < y \exists z\theta(\Param\uh_z, \param, x', y', z)$ be its corresponding weakly $\forall \Pi^0_4$ formula. 
$\RCA_0 \vdash S \rightarrow S'$. We now prove that $\RCA_0 + \BSig_2 \vdash S' \rightarrow S$.
Fix some~$\Param$ and some~$\param$. By $S'$, there is some~$x$ such that $\forall y \exists x' < x \forall y' < y \exists z \theta(\Param\uh_z, \param, x', y', z)$.
Let~$f : \NN \to x$ be defined by $f(y) = x'$  such that $\forall y' < y \exists z \theta(\Param\uh_z, \param, x', y', z)$. By $\RT^1$, which is equivalent to~$\BSig_2$, there is some infinite $f$-homogeneous set~$H$ for some color~$x' < x$. We claim that $\forall y' \exists z \theta(\Param\uh_z, \param, x', y', z)$.
Indeed, given~$y'$, there is some~$y > y'$ such that $y \in H$. Since~$f(y) = x'$, then by definition of~$f$, $\exists z \theta(\Param\uh_z, \param, x', y', z)$. This completes the proof.
\end{proof}

For the remainder of this section, fix a $\Sigma^0_0$-formula $\theta(x, y, z)$ (with parameters) and let~$T \equiv \forall x \exists y \forall x' < x \exists y' < y \forall z \theta(x', y', z)$.

\begin{proposition}\label[proposition]{prop:weak-largeness-bsig2-largeness}
For every~$n \in \omega$,
$\RCA_0 + T$ proves that for every~$k \geq 1$, $\bbomega^n \cdot k$-largeness$(T)$ is a largeness notion.
\end{proposition}
\begin{proof}
Exactly the proof of \Cref{prop:largeness-bsig2-largeness}, but the new form of $T$ removes the use of~$\BSig_2$.
\end{proof}

\begin{proposition}\label[proposition]{prop:conservation-t}
If $\WKL_0 + \BSig_2 \vdash \neg T$, then $\ISig_1 \vdash \neg T$.
\end{proposition}
\begin{proof}
Assume~$\WKL_0 + \BSig_2 \vdash \neg T$. 
Then $\WKL_0 + \BSig_2 + T \vdash \bot$.
By \Cref{thm:generalized-parsons-p4-bsig2-largeness} with~$\psi \equiv \bot$,
there is some~$n \in \omega$ such that $\ISig_1$ proves $\forall x \forall Z \finsub (x, \infty)$
$$
Z \mbox{ is } \bbomega^n\mbox{-large}(T) \mbox{ and exp-sparse } \rightarrow \bot
$$
Hence $\ISig_1$ proves that exp-sparse $\bbomega^n$-largeness$(T)$ is not a largeness notion, therefore, by \Cref{prop:weak-largeness-bsig2-largeness}, $\RCA_0 + T \vdash \bot$. By $\Pi_1^1$ conservation of $\RCA_0$ over $\ISig_1$ (see Friedman \cite{friedman1975systems}), $\ISig_1 \vdash \neg T$.

\end{proof}

\begin{proposition}\label[proposition]{prop:bsig2-weak-formulas}
$\WKL_0 + \BSig_2$ is conservative over~$\ISig_1$ for weakly $\forall \Pi^0_4$ formulas.
\end{proposition}
\begin{proof}[First proof of \Cref{prop:bsig2-weak-formulas}]
Fix a $\Sigma^0_0$-formula $\zeta(\Param \uh_z, \param, x, y, z)$, and let~$T(\Param, \param) \equiv \forall x \exists y \forall x' < x \exists y' < y \forall z \neg \zeta(\Param \uh_z, \param, x', y', z)$.
Suppose $\ISig_1 \not \vdash \forall \Param \forall \param \neg T(\Param, \param)$.
Then there is a model~$\M = (M, S) \models \ISig_1 \wedge T(\Param, \param)$ for some~$\Param \in S$ and $\param \in M$. Enrich the language with a constant symbol for~$\Param$ and $\param$. In particular, $\ISig_1 \not \vdash \neg T(\Param, \param)$, so by \Cref{prop:conservation-t},  $\WKL_0 + \BSig_2 \not \vdash \neg T(\Param, \param)$. It follows that there is a model $\mathcal{N} = (N, R) \models \WKL_0 + \BSig_2 \wedge T(\Param, \param)$,
hence $\mathcal{N} \models \WKL_0 + \BSig_2 \wedge \neg \forall \Param \forall \param \neg T(\Param, \param)$,
thus $\WKL_0 + \BSig_2 \not \vdash \forall \Param \forall \param \neg T(\Param, \param)$.
\end{proof}

We now give an alternative and more traditional proof of \Cref{prop:bsig2-weak-formulas} using the notion of cofinal $\Sigma^0_n$-elementary extension.

\begin{definition}
A model~$\mathcal{N}$ is a \emph{cofinal} extension of~$\M \subseteq \mathcal{N}$ (written $\M \subseteq_{\operatorname{cf}} \mathcal{N}$) if for every~$x \in \mathcal{N}$, there is some~$y \in \M$ such that $x \leq y$. We write $\M \preceq_{n, \operatorname{cf}} \mathcal{N}$ if $\mathcal{N}$ is a cofinal $\Sigma^0_n$-elementary extension of~$\M$.
\end{definition}

The following theorem first appeared in Paris~\cite{paris1980hierarchy}, and was independently discovered by Harvey Friedman.

\begin{theorem}[Paris~\cite{paris1980hierarchy} ; Friedman]\label[theorem]{thm:cofinal-extension}
Let~$n \in \omega$ and $\M \models \ISig_n$. Then there is $\mathcal{N} \succeq_{n+1, \operatorname{cf}} \M$ that satisfies $\BSig_{n+1}$.
\end{theorem}

We are now ready to give an alternative proof of \Cref{prop:bsig2-weak-formulas}.

\begin{proof}[Second proof of \Cref{prop:bsig2-weak-formulas}]
Fix a $\Sigma^0_0$-formula $\zeta(\Param\uh_z, \param, x, y, z)$, and let
$$T(\Param, \param) \equiv \forall x \exists y \forall x' < x \exists y' < y \forall z \neg \zeta(\Param\uh_z,\param,x', y', z)$$
Suppose $\ISig_1 \not \vdash \forall \Param \forall \param \neg T(\Param, \param)$.
Then there is a model~$\M = (M, S) \models \ISig_1 \wedge T(\Param, \param)$ for some~$\Param \in S$ and $\param \in M$.
Enrich the language with a constant symbol for~$\Param$ and $\param$.
By \Cref{thm:cofinal-extension}, there exists a model $\Nc = (N, U) \succeq_{2, \operatorname{cf}} \M$ of~$\BSig_2$. 

We claim that~$\Nc \models T(\Param, \param)$. Fix some~$x \in N$. By cofinality, let~$x_1 \in M$ be such that~$x \leq x_1$. Since~$\M \models T(\Param, \param)$, there is some~$y_1 \in M$ such that $\M \models \forall x' < x_1 \exists y' < y_1 \forall z \neg \zeta(\Param\uh_z,\param,x', y', z)$. Since~$\M \models \ISig_1$, the previous formula is equivalent to a $\Pi^0_1$ formula, so since~$\Nc$ is a $\Sigma^0_2$-elementary extension of~$\M$, $\Nc \models \forall x' < x_1 \exists y' < y_1 \forall z \neg \zeta(\Param\uh_z,\param,x', y', z)$. In particular, $\Nc \models \exists y \forall x' < x_1 \exists y' < y \forall z \neg \zeta(\Param\uh_z,\param, x', y', z)$, and this for every~$x \in N$. Thus $\Nc \models T(\Param,\param)$. It follows that $\BSig_2 \nvdash \forall \Param \forall \param \neg T(\Param, \param)$.

Last, By H\'ajek~\cite{hajek1993interpretability}, $\WKL_0 + \BSig_2$ is a $\Pi^1_1$ conservative extension of~$\BSig_2$, so $\WKL_0 + \BSig_2 \not \vdash \forall \Param \forall \param \neg T(\Param, \param)$.
\end{proof}

\begin{corollary}
    $\WKL_0 + \RT_2^2$ is weakly $\forall \Pi^0_4$ conservative over $\ISig_1$.
\end{corollary}

\begin{proof}
    Immediate from the combination of \Cref{thm:rt22-pi04-conservation} and \Cref{prop:bsig2-weak-formulas}.
\end{proof}

\section{Lower bounds for largeness$(T)$}\label[section]{sect:conclusion}

The proof structure of the $\Pi^0_4$ conservation theorem for Ramsey's theorem for pairs and two colors followed closely the one of Patey and Yokoyama~\cite{patey2018proof}. In particular, the generalized Parsons theorem for $\WKL_0 + \BSig_2 + T$ (\Cref{thm:generalized-parsons-p4-bsig2-largeness}) is used together with the $\Pi^1_1$ conservation of the grouping principle, to obtain a finitary version in terms of largeness$(T)$ without explicit bounds (\Cref{prop:fgp-apartness}). Ko\l odziejczyk and Yokoyama~\cite[Theorem 2.4]{kolo2020some} computed explicit bounds to the finite grouping principle in terms of largeness, by iterating a pigeonhole lemma~\cite[Lemma 2.2]{kolo2020some}. 


The remainder of this section is devoted to the proof that the new bound of \Cref{lem:largeness-rt1k} for the pigeonhole principle is optimal.

\begin{definition}
An \emph{$\bbomega^n$-decomposition} of an $\bbomega^n$-large set $X$ is a finite sequence of $\bbomega^{n-1}$-large subsets~$X_0 < \dots < X_a$ of $X \setminus \min X$ for some~$a \geq \min X-1$.
A finite set~$X$ is \emph{minimal} for $\bbomega^n$-largeness if it is $\bbomega^n$-large
and for every~$x \in X$, $X \setminus \{x\}$ is not $\bbomega^n$-large.
\end{definition}

Note that by regularity of largeness, for every~$x, n \in \NN$, there is some $y > x$ such that $[x, y]$ is minimal for $\bbomega^n$-largeness.

\begin{lemma}\label[lemma]{lem:minimal-unique}
Fix~$n > 0$. Let~$X$ be minimal for~$\bbomega^n$-largeness.
Then it admits a unique $\bbomega^n$-decomposition $X_0 < \dots < X_a$.
Moreover, $a = \min X - 1$, $\{\min X\} \cup X_0 \cup \dots X_a = X$ and for every~$i < a$, $X_i$ is minimal for $\bbomega^{n-1}$-largeness.
\end{lemma}
\begin{proof}

Let $X$ be minimal for $\bbomega^n$-largeness for some~$n > 0$. Let~$X_0 < \dots < X_a$ be an $\bbomega^n$-decomposition of~$X$.

First, we claim that $a = \min X - 1$. Indeed, if $a > \min X-1$, then $X \setminus X_a$ would be $\bbomega^n$-large, contradicting minimality of $X$ for $\bbomega^n$-largeness.

Second, we claim that $X_i$ is minimal for $\bbomega^{n-1}$-largeness. Suppose not. Let~$x \in X_i$ be such that $X_i \setminus \{x\}$ is $\bbomega^{n-1}$-large. Then $X \setminus \{x\}$ would be $\bbomega^n$-large, again contradicting minimality of~$X$ for $\bbomega^n$-largeness.

Third, let us prove that $\{\min X\} \cup X_0 \cup \dots X_a = X$. Suppose there is some~$x \in X \setminus (\{\min X\} \cup X_0 \cup \dots X_a)$. Then once again, $X \setminus \{x\}$ would be $\bbomega^n$-large.

Last, assume there is another $\bbomega^n$-decomposition $Y_0 < \dots < Y_a$ of~$X$.
Then, by induction over~$i \leq a$, we prove that $X_i = Y_i$. Indeed, assuming that $X_j = Y_j$ for every~$j < i$, since $X = \{\min X\} \cup X_0 \cup \dots \cup X_a = \{\min X\} \cup Y_0 \cup \dots \cup Y_a$,
then either $X_i \subseteq Y_i$, or $Y_i \subseteq X_i$. By minimality of $X_i$ and $Y_i$ for $\bbomega^{n-1}$-largeness, $X_i = Y_i$. This completes the proof of \Cref{lem:minimal-unique}.
\end{proof}

The previous lemma justifies the following definition:

\begin{definition}
Fix a set~$X$ which is minimal for $\bbomega^n$-largeness.
The canonical $n$-block of~$X$ is $X$ itself.
For $c < n$, the canonical $c$-blocks of~$X$ are the canonical $c$-blocks of the sets $X_i$, where  $X_0 < \dots < X_{\min X-1}$ is the unique $\bbomega^n$-decomposition of~$X$.
\end{definition}

\begin{definition}
Let $X$ minimal for $\bbomega^n$-largeness for some $n \geq 1$.
The \emph{$0$-blockfree} subset of~$X$ is the set $X$ minus all the elements belonging to a canonical $0$-block of~$X$.
\end{definition}

\begin{lemma}\label[lemma]{lem:minus-0-blocks}
Let $X$ minimal for $\bbomega^n$-largeness for some $n \geq 1$. Its $0$-blockfree subset is a minimal $\bbomega^{n-1}$-large subset of $X$.   
\end{lemma}

\begin{proof}
Proceed by induction on $n$:

Case $n = 1$: Let $X$ be $\bbomega^1$-large, then its $0$-blockfree subset is equal to $\{\min X\}$ which is minimal $\bbomega^0$-large.

Case $n \geq 2$: Assume the property to be true at rank $n - 1$, let $X_0 < \dots < X_{\min X - 1}$ be the canonical $\bbomega^{n}$-decomposition of $X$ into $\bbomega^{n-1}$-large sets. Let $X_i'$ be the $0$-blockfree subset of~$X_i$ for every $i < \min X$. By the inductive hypothesis, every $X_i'$ is a minimal $\bbomega^{n-2}$-large subset of $X_i$, therefore the $0$-blockfree subset $X'$ of~$X$ is $\bbomega^{n-1}$-large and $X_0' < \dots < X_{\min X - 1}'$ is the canonical $\bbomega^{n-1}$-decomposition of $X'$ into $\bbomega^{n-2}$-large sets.
\end{proof}

\begin{definition}
Let $X$ be minimal for $\bbomega^n$-largeness. Consider $\phi_X(x,y,c)$ a $\Sigma_0$ formula that is true if and only if $x$ and $y$ are in the same canonical $c$-block of $X$ for $c \leq n$. 

Let $\theta_X(x,y,z)$ be the following $\Sigma_0$ formula:
$$(x \in X \wedge z \in X \wedge z \geq y) \rightarrow \exists c \leq n, (\phi_X(y,z,c) \wedge \neg \phi_X(x,y,c) \wedge y > x \wedge y \in X)$$
Let $T_X$ be the formula $\forall x \exists y \forall z \theta_X(x,y,z)$. 
\end{definition}

All of these formulas are $\Sigma_0$, since there are only finitely many elements of $X$ to consider. Note that~$n$ is uniquely determined by~$X$ : it is the largest integer less or equal to~$\max X$ such that $X$ is $\bbomega^n$-large. Being $\bbomega^n$-large is a $\Sigma_0$ predicate in~$X$ and $n$, so the formula $\phi_X$ is $\Sigma_0$ uniformly in~$X$.

\begin{lemma}\label[lemma]{lem:apart-if-theta-bounds}
Let~$X$ be minimal for $\bbomega^n$-largeness for some~$n \geq 1$.
For every subsets~$A < B$ of~$X$, $A$ and $B$ are $T_X$-apart if and only if $\theta_X(\max A, \min B, \max B)$ holds.
\end{lemma}
\begin{proof}
If $A$ and $B$ are $T_X$-apart, then there exists some $y \leq \min B$ such that $\theta_X(\max A, y, \max B)$. Since $\max A \in X$, $\max B \in X$ and $\max B \geq y$, unfolding the definition of $\theta_X$ yield $y \in X$, $y > \max A$ and there exists some $c \leq n$ such that $\phi_X(y, \max B,c)$ and $\neg \phi_X(\max A, y,c)$ holds. Since $\min B \in [y,\max B]$ and $\min B \in X$, $\phi_X(\min B, \max B,c)$ also holds, and since $y \in [\max A, \min B]$, $\neg \phi_X(\max A, \min B,c)$ holds. So $\theta_X(\max A, \min B, \max B)$ holds. \\

Conversely, if $\theta_X(\max A, \min B, \max B)$ holds: since $\max A \in X$, $\max B \in X$ and $\max B \geq \min B$, there exists some $c \leq n$ such that $\phi_X(\min B, \max B,c)$ and $\neg \phi_X(\max A, \min B,c)$ holds. To show that $A$ and $B$ are $T_X$-apart, it is sufficient to show that $\theta_X(x,\min B, z)$ holds for every $x \leq \max A$ and $z \leq \max B$. If $x \notin X$ or $z \notin X$ or $z < \min B$ then $\theta_X(x,\min B,z)$ holds, so assume $x \in X$ and $z \in X$ and $\min B \leq z$. Since $x \leq \max A$ and $\neg \phi_X(\max A, \min B,c)$ holds, then $\neg \phi_X(x, \min B,c)$ holds and since $z \leq \max B$ and $\phi_X(\min B, \max B,c)$ holds, then $\phi_X(\min B, z,c)$ holds. Therefore, $\forall x \leq \max A \exists y \leq \min B \forall z \leq \max A ~\theta_X(x,y, z)$ holds. This completes the proof of \Cref{lem:apart-if-theta-bounds}.
\end{proof}

\begin{lemma}\label[lemma]{lem:theta-equiv-subblock}
Let $X$ be minimal for $\bbomega^n$-largeness for some $n \geq 1$. Let $Y$ be a canonical $c$-block of $X$ for some $c \leq n$. Then $\theta_X$ and $\theta_Y$ are equivalent for elements of $Y$.  
\end{lemma}

\begin{proof}
Let $x,y \in Y$. Then for every $d \leq c$, $\phi_X(x,y,d)$ holds if and only if $\phi_Y(x,y,d)$ holds as the canonical $d$-block of $Y$ are exactly the canonical $d$-block of $X$ contained in $Y$.

Let $x,y,z \in Y$, if $z < y$ then $\theta_X(x,y,z)$ and $\theta_Y(x,y,z)$ hold. So assume $y \leq z$, $\theta_X(x,y,z)$ holds if and only if there exists some $d \leq n$ such that $\phi_X(y,z,d)$ and $\neg \phi_X(x,y,d)$ and $y > x$. Since $\phi_X(x,y,c)$ holds (as $Y$ is a canonical $c$-block) then $d < c$. So $\theta_X(x,y,z)$ holds if and only if $\phi_Y(y,z,d)$ and $\neg \phi_Y(x,y,d)$ and $y > x$ which is equivalent to $\theta_Y(x,y,z)$. This completes the proof of \Cref{lem:theta-equiv-subblock}.
\end{proof}

Combining \Cref{lem:apart-if-theta-bounds} and \Cref{lem:theta-equiv-subblock} yield the following corollaries:

\begin{corollary}\label[corollary]{cor:largeness-equiv-subblock}
Let $X$ be minimal for $\bbomega^n$-largeness for some $n \geq 1$ and $Y$ a canonical $c$-block of $X$ for some $c \leq n$.
\begin{enumerate}
    \item Two subsets $A < B$ of $Y$ are $T_Y$-apart, if and only if they are $T_X$-apart.
    \item For $k \leq c$, the $\bbomega^k$-large$(T_Y)$ subsets of $Y$ are exactly the $\bbomega^k$-large$(T_X)$ subsets of~$Y$.
\end{enumerate}    
\end{corollary}

\begin{lemma}\label[lemma]{lem:theta-equiv-minus-0-block}
Let $X$ be minimal for $\bbomega^n$-largeness for some $n \geq 1$. Let $X'$ be the $0$-blockfree subset of~$X$. Then $\theta_X$ and $\theta_{X'}$ are equivalent for elements of $X'$.  
\end{lemma} 

\begin{proof}
Let $x,y \in X'$. Then for every $1 \leq c \leq n$, $\phi_X(x,y,c)$ holds if and only if $\phi_{X'}(x,y,c-1)$ holds, as the canonical $(c-1)$-blocks of $X'$ are exactly the canonical $c$-block of $X$ minus the elements belonging to a canonical $0$-block.  

Let $x,y,z \in X'$, if $z < y$ then $\theta_X(x,y,z)$ and $\theta_{X'}(x,y,z)$ hold. So assume $y \leq z$, $\theta_X(x,y,z)$ holds if and only if there exists some $c \leq n$ such that $\phi_X(y,z,c)$ and $\neg \phi_X(x,y,c)$ and $y > x$. Since $x,y \in X'$, $c$ cannot be equal to $0$. So $\theta_X(x,y,z)$ holds if and only if $y > x$ and $\phi_{X'}(y,z,c-1)$ and $\neg \phi_{X'}(x,y,c-1)$ holds, which is equivalent to $\theta_{X'}(x,y,z)$. This completes the proof of \Cref{lem:theta-equiv-minus-0-block}.
\end{proof}

Combining \Cref{lem:apart-if-theta-bounds} and \Cref{lem:theta-equiv-minus-0-block} yield the following corollaries:

\begin{corollary}\label[corollary]{cor:largeness-equiv-minus-0-block}
Let $X$ be minimal for $\bbomega^n$-largeness for some $n \geq 1$ and $X'$ be its $0$-blockfree subset.
\begin{enumerate}
    \item Two subsets $A < B$ of $X'$ are $T_{X'}$-apart, if and only if they are $T_X$-apart.
    \item For $k \leq n-1$, the $\bbomega^k$-large$(T_{X'})$ subsets of $X'$ are exactly the $\bbomega^k$-large$(T_X)$ subsets of~$X'$.
\end{enumerate}    
\end{corollary}

\begin{lemma}\label[lemma]{lem:canonical-is-apart}
Let~$X$ be minimal for $\bbomega^n$-largeness for some~$n \geq 1$.
Then $X$ is $\bbomega^n$-large$(T_X)$.
\end{lemma}
\begin{proof}
For this, it suffices to prove that for every~$a < n$, any two canonical $a$-blocks $Y < Z$ of $X$ are $T_X$-apart.
Let $Y < Z$ be two such blocks. Then $\phi_X(\min Z,\max Z,a)$ holds since~$\min Z$ and $\max Z$ belong to the same $a$-block, and $\phi_X(\max Y, \min Z, a)$ does not hold since $\max Y$ is not ~in $Z$. It follows that $\theta_X(\max Y, \min Z, \max Z)$ holds, so by \Cref{lem:apart-if-theta-bounds}, $Y$ and $X$ are $T_X$-apart. This completes the proof of \Cref{lem:canonical-is-apart}.
\end{proof}

\begin{lemma}\label[lemma]{lem:apart-included}
Let~$X$ be minimal for $\bbomega^n$-largeness for some~$n \geq 1$.  
Let $X_0 < \dots < X_{\min X - 1}$ be the canonical $\bbomega^n$-decomposition of $X$ into $\bbomega^{n-1}$-large sets.
If $A < B$ are two $T_X$-apart subsets of $X$, then $B \subseteq X_i$ for some $i$.
\end{lemma}
\begin{proof}
Since $A$ and $B$ are $T_X$-apart, then by \Cref{lem:apart-if-theta-bounds}, $\theta_X(\max A, \min B, \max B)$ holds. Since $\max A \in X$, $\max B \in X$ and $\max B \geq \min B$, unfolding the definition of~$\theta_X$ yield that there exists some $c \leq n$ such that $\phi_X(\min B,\max B,c)$ and $\neg \phi_X(\max A,\min B,c)$ holds. Since $\phi_X(\max A,\min B,n)$ holds (as $\max A, \min B \in X$), $c$ cannot be equal to $n$. On the other hand, $\min B$ and $\max B$ are in the same canonical $c$-block of $X$, so there are in the same canonical $(n-1)$-block. So $B$ is included in $X_i$ for some $i$. This completes the proof of \Cref{lem:apart-included}.
\end{proof}


In the following proposition, recall that if~$X$ is minimal for $\bbomega^{2n-1}$-largeness, then it is $\bbomega^{2n-1}$-large($T_X$) by \Cref{lem:canonical-is-apart}. Thus, the proposition gives us an example of an $\bbomega^{2n-1}$-large$(T_X)$ set which is not $\bbomega^n$-large$(T_X, \RT^1_2)$.

\begin{proposition}\label[proposition]{prop:rt12-lower-bound}
Let~$X$ be minimal for $\bbomega^{2n-1}$-largeness for some~$n \geq 1$.  
There exists a coloring $f_X : X \to 2$ such that there is no $f_X$-homogeneous $\bbomega^n$-large$(T_X)$ subset of $X$.
\end{proposition}
\begin{proof}
Let $f_X$ be the following 2-coloring of $X$: for $x \in X$, consider the smallest $c$ such that $x$ is in a $c$-canonical block of $X$, then color $x$ with the parity of $c$. \\


Proceed by induction on $n$:

\textbf{Case $n = 1$}: Let $X$ be minimal for $\bbomega$-largeness (in other words, $|X| = \min X + 1$). 
Then $f_X(\min X) = 1$ and $f_X(x) = 0$ for every other element of $X$. So there is no $f_X$-homogeneous $\bbomega$-large$(T_X)$ subset of $X$ by minimality of $X$.
\bigskip

\textbf{Case $n$}: Assume the property to be true at rank $n-1$, and let $X$ be minimal for $\bbomega^{2n-1}$-largeness. Let $X_0 < \dots < X_{\min X - 1}$ be the canonical $\bbomega^{2n-1}$-decomposition of $X$ into $\bbomega^{2n-2}$-large sets.

Assume by contradiction that there exists an $f_X$-homogeneous $\bbomega^n$-large$(T_X)$ subset $Y \subseteq X$, we can assume $Y$ to be minimal for $\bbomega^n$-largeness$(T_X)$.  Let $Y_0 < \dots < Y_{\min Y - 1}$ be the canonical $\bbomega^n$-decomposition of $Y$ into $\bbomega^{n-1}$-large sets. 

There are two cases:
\begin{itemize}
    \item $Y$ is $f_X$ homogeneous for the color $0$: Since $f_X(\min X) = 1$, $\min X \notin Y$, so $\min Y > \min X$. By the finite pigeonhole principle, there exists some $i < \min X$, such that $X_i$ contains two elements of the form $\min Y_j$ for some $j < \min Y$. Since $Y_0 < \dots < Y_{\min Y -1}$, we can assume these two elements to be of the form $\min Y_j$, $\min Y_{j+1}$, so there is some $j < \min Y - 1$ such that $Y_j \subseteq X_i$ and $\min Y_{j+1} \in X_i$. By \Cref{lem:apart-included}, since $Y_j$ and $Y_{j+1}$ are $T_X$-apart, $Y_{j+1}$ must also be included in $X_i$. \\

    Let $X_{i,0} < \dots < X_{i,\min X_i - 1}$ be the canonical $\bbomega^{2n-2}$-decomposition of $X_i$ into $\bbomega^{2n-3}$-large sets. By \Cref{cor:largeness-equiv-subblock}, since $Y_j, Y_{j+1}$ are $T_X$-apart and subsets of $X_i$, they are also $T_{X_i}$-apart. So, by \Cref{lem:apart-included} applied to $X_i$, $Y_{j+1} \subseteq X_{i,i'}$ for some $i' < \min X_i - 1$. By \Cref{cor:largeness-equiv-subblock}, since $Y_{j+1}$ is $\bbomega^{n-1}$-large$(T_X)$ then it is $\bbomega^{n-1}$-large$(T_{X_{i,i'}})$. But $f_X \uh X_{i, i'}$ is equal to $f_{X_{i, i'}}$, so $Y_{j+1}$ is an $\bbomega^{n-1}$-large$(T_{X_{i,i'}})$, $f_{X_{i, i'}}$-homogeneous subset of $X_{i,i'}$, contradicting the induction hypothesis.

    \item $Y$ is $f_X$ homogeneous for the color $1$: Let $X'$ be the $0$-blockfree subset of $X$. Then, by \Cref{lem:minus-0-blocks}, $X'$ is a minimal $\bbomega^{2n-2}$-large subset of $X$ and its canonical decomposition $X_0' < \dots < X_{\min X - 1}'$ satisfies that $X_i'$ is equal to $X_i$ minus all the elements belonging to a canonical $0$-block. Furthermore, $Y \subseteq X'$ (since all the elements belonging to a canonical $0$-block of $X$ are $0$-colored).

    If $\max Y_0 \in X_0$ then $Y_0 \subseteq X_0$, and otherwise, by the finite pigeonhole principle, there exists some $1 \leq i < \min X$ such that $X_i$ contains two elements of the form $\max Y_j$. Since $Y_0 < \dots < Y_{\min Y - 1}$, in both cases we have $Y_j \subseteq X_i$ for some $j < \min Y$ and $i < \min X$ and therefore $Y_j \subseteq X_i'$. By \Cref{cor:largeness-equiv-minus-0-block}, $Y_j$ is an $\bbomega^{n-1}$-large$(T_{X_i'})$ subset of $X_i'$. $f_X \uh X_i'$ is equal to $1 - f_{X_i'}$ (a canonical $c$-block of $X$ became a canonical $(c-1)$-block of $X'$ when we get rid of the $0$-canonical elements), so $Y_j$ is an $\bbomega^{n-1}$-large$(T_{X_i'})$, $f_{X_i'}$-homogeneous (for the color $0$) subset of $X_i'$, contradicting the induction hypothesis.

    \end{itemize}
We conclude by induction. This completes the proof of \Cref{prop:rt12-lower-bound}.
\end{proof}

\section{Open questions}\label[section]{sect:open-questions}

There exists a close connection between explicit bounds computation for largeness, and proof speedup theorems. In particular Ko\l odziejczyk, Wong and Yokoyama~\cite{kolodziejczyk2023ramsey} proved that Ramsey's theorem for pairs and two colors has at most polynomial speedup over~$\RCA_0$ for $\forall \Pi^0_3$ sentences, using the fact that $\bbomega^{300n}$-largeness is sufficient to obtain a homogeneous $\bbomega^n$-large set for every instance of~$\RT^2_2$. The exponential bounds of largeness$(T)$ for Ramsey's theorem for pairs yields the following natural questions:

\begin{question}\label[question]{ques:speedup}
Does $\RT^2_2$ admit exponential proof speedup over~$\RCA_0 + \BSig_2$?
\end{question}

\begin{question}
Is there a polynomial~$p$ such that for every~$n$, 
every~$\bbomega^{p(n)}$-large$(T)$ set is $\bbomega^n$-large$(T, \RT^2_2)$?
\end{question}

The lower bound for the pigeonhole principle uses a formula~$T_X$ which depends on the considered set~$X$. 

\begin{question}
Is there a $\Pi^0_3$ formula $T$ such that for every~$n \in \omega$, every $\bbomega^n$-large$(T,\RT^1_2)$ set is $\bbomega^{2n-1}$-large$(T)$?
\end{question}

We note that Question~\ref{ques:speedup} is essential for the original question on $\Pi^1_1$-conservation (Question~\ref{ques:rt22-pi11-conservation}). By the discussion of \cite[Section 5]{fiori2021isomorphism} and an upcoming paper by Ikari, Ko{\l}odziejczyk and Yokoyama (see \cite{Ikari-PhD}), if $\RCA_0 + \RT^2_2$ is $\Pi^0_5$-conservative over $\RCA_0 + \BSig_2$, then there exists a poly-time proof transformation between $\RCA_0 + \RT^2_2$ and $\RCA_0 + \BSig_2$ for $\Pi^1_1$-consequences.
Thus, a positive answer to Question~\ref{ques:speedup} implies that our $\Pi^0_4$-conservation is the best possible.

\begin{center}
\textbf{Acknowledgement}
\end{center}
The authors are thankful to Leszek Ko{\l}odziejczyk for insightful comments and discussions and for the anonymous referee for his careful reading and improvement suggestions.
The work of the third author is partially supported by
 JSPS KAKENHI grant numbers JP19K03601, JP21KK0045 and JP23K03193.
 
\bibliographystyle{plain}
\bibliography{biblio.bib}

\end{document}